\documentclass{amsart}

\usepackage{amsmath}
\usepackage{amssymb}
\usepackage{amsthm}
\usepackage{xcolor}
\usepackage[utf8]{inputenc}
\usepackage{tikz}
\usepackage{pgfplots}
\usepackage[hidelinks]{hyperref}
\usepackage[nameinlink,noabbrev]{cleveref}
\usepackage{subcaption}
\usepackage{booktabs}
\usepackage{array}
\newcolumntype{L}[1]{>{\raggedright\let\newline\\\arraybackslash\hspace{0pt}}m{#1}}
\newcolumntype{C}[1]{>{\centering\let\newline\\\arraybackslash\hspace{0pt}}m{#1}}
\newcolumntype{R}[1]{>{\raggedleft\let\newline\\\arraybackslash\hspace{0pt}}m{#1}}

\definecolor{myBlue}{RGB}{30,144,255} 
\definecolor{myGreen}{RGB}{69,169,0} 
\definecolor{myRed}{RGB}{165,12,42} 
\definecolor{myOrange}{RGB}{225,92,22} 

\newcommand{\cC}{\mathcal{C}}
\newcommand{\cQ}{\mathcal{Q}}
\newcommand{\cL}{\mathcal{L}}
\newcommand{\cS}{\mathcal{S}}
\newcommand{\cX}{\mathcal{X}}
\newcommand{\cT}{\mathcal{T}}
\newcommand{\dd}{\textrm{d}}
\newcommand{\hb}{\hat{b}}
\newcommand{\he}{\hat{e}}
\newcommand{\be}{\bar{e}}
\newcommand{\bE}{\bar{E}}

\newcommand{\mR}{\mathbb{R}}

\newcommand{\tu}{\tilde{u}}
\newcommand{\tU}{\tilde{U}}

\newcommand{\vx}{\vct{x}}

\newcommand{\vvf}{\vvct{f}}
\newcommand{\vvx}{\vvct{x}}
\newcommand{\vvy}{\vvct{y}}
\newcommand{\vvu}{\vvct{u}}
\newcommand{\dx}{\;\dd\vx}
\newcommand{\dt}{\;\dd t}

\def\Nb{\mathrm{N}}

\newcommand{\vct}[1]{\mathbf{#1}}
\newcommand{\vvct}[1]{\underline{#1}}

\newtheorem{thm}{Theorem}[section]
\newtheorem{lem}[thm]{Lemma}
\newtheorem{cor}[thm]{Corollary}
\newtheorem{rem}[thm]{Remark}

\makeatletter
\newcommand\addplotgraphicsnatural[2][]{%
	\begingroup
	\pgfqkeys{/pgfplots/plot graphics}{#1}%
	\setbox0=\hbox{\includegraphics{#2}}%
	\pgfmathparse{\wd0/(\pgfkeysvalueof{/pgfplots/plot graphics/xmax} - \pgfkeysvalueof{/pgfplots/plot graphics/xmin})}%
	\let\xunit=\pgfmathresult
	\pgfmathparse{\ht0/(\pgfkeysvalueof{/pgfplots/plot graphics/ymax} - \pgfkeysvalueof{/pgfplots/plot graphics/ymin})}%
	\let\yunit=\pgfmathresult
	\xdef\marshal{%
		\noexpand\pgfplotsset{unit vector ratio={\xunit\space \yunit}}%
	}%
	\endgroup
	\marshal
	\addplot graphics[#1] {#2};
}   
\makeatother

\begin{document}

\title[Fast Mass Lumped Multiscale Wave Propagation Modelling]{Fast Mass Lumped Multiscale Wave Propagation Modelling}
\author[S.~Geevers and R.~Maier]{Sjoerd Geevers$^\dagger$, Roland Maier$^\ddagger$}
\date{\today}
\address{${}^{\dagger}$ Faculty of Mathematics, University of Vienna, Oskar-Morgenstern-Platz 1, 1090 Vienna, Austria}
\email{sjoerd.geevers@univie.ac.at}
\address{${}^{\ddagger}$ Department of Mathematical Sciences, Chalmers University of Technology and University of Gothenburg, Chalmers Tv\"argata 3, 412 96 G\"oteborg, Sweden}
\email{roland.maier@chalmers.se}

%
%
\begin{abstract}
In this paper, we investigate the use of a mass lumped fully explicit time stepping scheme for the discretisation of the wave equation with underlying material parameters that vary at arbitrarily fine scales.
We combine the leapfrog scheme for the temporal discretisation with the multiscale technique known as Localized Orthogonal Decomposition for the spatial discretisation.
To speed up the method and to make it fully explicit, a special mass lumping approach is introduced that relies on an appropriate interpolation operator. 
This operator is also employed in the construction of the Localized Orthogonal Decomposition and is a key feature of the approach. 
We prove that the method converges with second order in the energy norm, with a leading constant that does not depend on the scales at which the material parameters vary.
We also illustrate the performance of the mass lumped method in a set of numerical experiments.
\end{abstract}

\maketitle



\section{Introduction}

Explicit finite element time stepping schemes are popular methods for wave propagation modelling due to their efficiency, scalability, and ease of implementation. However, the time step size of such schemes is dictated by the size of the smallest element in the spatial discretisation. This condition becomes very restrictive when the spatial material parameters vary at a very fine scale, since classical finite element require a very fine mesh to resolve the information at these fine scales. 

Multiscale methods can circumvent this problem by preprocessing the spatial material parameters and/or the finite element approximation spaces in a sophisticated manner. The actual simulation can then be performed accurately on a coarse mesh, with a mesh size that no longer depends on the fine scales of the spatial parameters. This significantly reduces the number of degrees of freedom and number of time steps of the actual simulation.

In the context of multiscale wave propagation, the Heterogeneous Multiscale Method (HMM) \cite{EE03,EE05} has been used for the spatial discretisation in~\cite{AbdG11,EngHR11} and in~\cite{EngHR12,AbdGS14,ArjR14}, where the latter focus on long-time wave propagation.
However, the HMM is an example of a multiscale method based on analytical homogenisation theory and therefore requires structural assumptions such as scale separation and periodicity. 
To avoid these assumptions, so-called numerical homogenisation approaches emerged; see, e.g.,~\cite{AltHP21} for an overview. These methods do not rely on results from homogenisation theory and provably work for minimal structural assumptions. In the context of the wave equation, this has for instance been examined in~\cite{OwhZ08}. The approach therein, however, requires a set of assumptions (Cordes-type conditions) that are hard to verify and 
the method has a large computational overhead due to the global auxiliary (but time-independent) problems that need to be solved. Similarly, also the approach in~\cite{JiaEG10,JiaE12} is based on auxiliary global problems that can be precomputed. 
The method uses the Multiscale Finite Element Method (MsFEM) introduced in~\cite{HowW97}, which constructs a set of multiscale basis functions with prescribed values on element boundaries. 
The method in~\cite{OwhZB14} allows for localization of the auxiliary problems, but these local problems involve more complex biharmonic equations. 
Based on an interpretation of numerical homogenisation within a game theoretical approach, the use of so-called gamblets has been considered in~\cite{OwhZ17} in a multi-level setting. 
Yet another multiscale method in~\cite{ChuEL14} relies on a generalized MsFEM which is based on the solution of local spectral problems and a discontinuous Galerkin ansatz. 
An overview of different multiscale methods in connection with the wave equation is presented in~\cite{AbdH17b}.

In the present work, we base our construction on the multiscale method known as Localized Orthogonal Decomposition (LOD)~\cite{MalP14,HenP13,MalP20}, which constructs an adapted approximation space based on the solution of localized elliptic sub-problems and provably works under minimal structural assumptions on 
the spatial material parameters. 
The LOD has been successfully applied to time-harmonic wave propagation problems, such as the Helmholtz equation; see e.g.,~\cite{GalP15,Pet17,PetV20,MaiV20}. Further, the approach has been studied for the wave equation with an implicit time discretisation~\cite{AbdH17} and also in connection with an explicit leapfrog method~\cite{PetS17,MaiP19}. A very favourable side effect for the combination of the LOD with an explicit scheme is the relaxation of the time step restriction which leads to a complexity reduction in both space and time. 

Here, we start from these latter ideas and aim at further reducing the computation costs by the use of a suitable lumping strategy. 
The overall goal hereby is to keep optimal convergence rates with a rather relaxed time step restriction and a significant speed-up due to the replacement of the mass matrix by a suitable diagonal approximation. 
A common technique to lump the mass matrix into a diagonal matrix is by taking the sum of each row. 
Another technique is to evaluate the mass matrix entries using a quadrature rule where the quadrature points coincide with the finite element nodes. These techniques are often equivalent. Mass lumping techniques for the wave equation have been studied for quadrilateral and hexahedral elements \cite{seriani94,komatitsch98}, also known as spectral elements, triangular elements \cite{cohen95,cohen01,mulder96,chin99,mulder13,cui17,liu17}, and tetrahedral elements \cite{mulder96,chin99,geevers18}. A rigorous error analysis and conditions for when mass lumping preserves optimal convergence rates are presented in \cite{geevers18,geevers19}. The analysis, however, fails when the mass matrix involves a spatial material parameter that varies at a scale much finer than the mesh size. We therefore introduce here an adapted interpolation operator that is then employed for the lumping strategy as well as in the construction of the LOD method such that the resulting lumped multiscale approach maintains a second-order convergence rate in the energy norm.

The rest of the paper is organised as follows. In Section~\ref{s:fem}, we introduce the heterogeneous wave equation and discuss a classical finite element approximation and a standard mass lumping strategy, which are only reliable if fine oscillations of the material parameters are resolved. We therefore introduce a mass lumped multiscale approach in Section~\ref{s:multiscale} based on the LOD method and a leapfrog discretisation in time. The error of the method is rigorously analysed in Section~\ref{s:error}, which marks the main contribution of this work. We then present numerical illustrations in Section~\ref{s:numerics}, before we summarise our conclusions in Section~\ref{s:conclusion}.\\

\noindent\textbf{Notation.} Throughout this work, we use the following notation. We let $L^2(\Omega)$ denote the Lebesgue space of square-integrable functions on the domain $\Omega$, $H^k(\Omega)$ the Sobolev space of functions in $L^2(\Omega)$ of which all partial derivatives up to order $k$ are also in $L^2(\Omega)$, $H^1_0(\Omega)$ the space of functions in $H^1(\Omega)$ with zero trace on the boundary $\partial\Omega$, and $L^{\infty}(0,T;X)$, with $X$ a Banach space, the Bochner space of functions $u\colon(0,T)\rightarrow X$ such that $\|u(\cdot,t)\|_X$ is uniformly bounded on $(0,T)$. We abbreviate $\|\cdot\|_0 = \|\cdot\|_{L^2(\Omega)}$, $\|\cdot\|_k = \|\cdot\|_{H^k(\Omega)}$, as well as $\|\cdot\|_{\infty,k} = \|\cdot\|_{L^\infty(0,T;H^k(\Omega))}$. Restrictions of the norms to subdomains $S \subseteq \Omega$ are denoted with $\|\cdot\|_{k;S}$ and $\|\cdot\|_{\infty,k;S}$. The corresponding semi-norms are denoted by $|\cdot|_{\bullet}$. Further, we use $C$ for some positive constant that does not depend on the mesh widths $H,\,h$ or the fine scale on which the 
spatial coefficients $\alpha$ and $\beta$ (which represent the material parameters) vary, but may depend on the diameter of the domain $\Omega$, the bounds $\alpha_{\min}$, $\alpha_{\max}$, $\beta_{\min}$, and $\beta_{\max}$, and the regularity of the mesh.
Note that $C$ might change from line to line in the estimates.

\section{Classical finite element method for the wave equation}\label{s:fem}

\subsection{Wave equation}
Let $\Omega\subset\mR^d$ be a bounded polyhedral domain in $d$ dimensions with a Lipschitz boundary $\partial\Omega$ and let $(0,T)$ be a time interval. We consider the wave equation
\begin{subequations}\label{eq:exactSol}
\begin{align}
\partial_t^2 u &= \beta^{-1} \nabla\cdot(\alpha\nabla u) + f \qquad&&\text{in }\Omega\times(0,T), \label{eq:exactSol1a} \\
u &= 0 &&\text{on }\partial\Omega\times(0,T), \\
u|_{t=0 } &= 0 &&\text{in }\Omega, \\
\partial_tu|_{t=0} &= 0 &&\text{in }\Omega,
\end{align}
\end{subequations}
where $u=u(\vx,t)$ is the wave field that needs to be computed, $\partial_t$ denotes the time derivative, $\nabla=(\partial_1,\partial_2,\dots,\partial_d)^t$ denotes the vector of partial derivatives $\partial_i$ with respect to direction $x_i$ (i.e., $\nabla$ is the gradient operator and $\nabla\cdot$ is the divergence operator), $\beta=\beta(\vx):0<\beta_{\min}<\beta(\vx)<\beta_{\max}<\infty$ and $\alpha=\alpha(\vx):0<\alpha_{\min}<\alpha(\vx)<\alpha_{\max}<\infty$ denote bounded spatial coefficients, and $f=f(\vx,t)$ denotes the source term.

\begin{rem}
The acoustic wave equation can be written in the form of \eqref{eq:exactSol}. In that case, $u$ denotes the pressure field and we have $\alpha=\rho^{-1}$ and $\beta=\rho^{-1}c^{-2}$, where $\rho$ denotes the mass density and $c$ the wave propagation speed.
\end{rem}

Assuming that $f\in L^\infty(0,T;L^2(\Omega))$, the weak formulation of~\eqref{eq:exactSol} can be stated as follows: find $u\in L^\infty(0,T;H^1_0(\Omega))$, with $\partial_t u\in L^\infty(0,T;L^2(\Omega))$ and $\partial_t(\beta\partial_tu)\in L^\infty(0,T;H^{-1}(\Omega))$, such that $u|_{t=0}=\partial_tu|_{t=0}\equiv 0$ and such that
\begin{align}
\label{eq:WF}
\langle \partial_t(\beta\partial_t u), w \rangle + (\alpha\nabla u,\nabla w) = (\beta f,w) \qquad \text{for all $w\in H_0^1(\Omega)$ and a.e. $t\in(0,T)$,}
\end{align}
where $\langle\cdot,\cdot\rangle$ denotes the pairing of distributions in $H^{-1}(\Omega)$ and functions in $H^1_0(\Omega)$, and $(\cdot,\cdot)$ denotes the $L^2(\Omega)$ (or $L^2(\Omega)^d$) inner product. 
Well-posedness of this problem follows from the proof of \cite[Theorem 8.1, Remark 8.2]{LioM72} by replacing the inner product $(\cdot,\cdot)$ by the weighted $L^2(\Omega)$ inner product $(\beta\cdot,\cdot)$.

If $\partial_tf\in L^\infty(0,T;L^2(\Omega))$ and $f|_{t=0}\equiv 0$, then $\partial_t^2u\in L^{\infty}(0,T;L^2(\Omega))$; see Lemma \ref{lem:regEst}. We can then rewrite \eqref{eq:WF} as
\begin{align}
\label{eq:WF2}
b(\partial_t^2u, w) + a( u, w) = b(f,w) \qquad \text{for all $w\in H_0^1(\Omega)$ and a.e. $t\in(0,T)$,}
\end{align}
where
\begin{align*}
b(u,w)&:= \int_{\Omega} \beta u w \;\dd\vx, \qquad a(u,w) := \int_{\Omega} \alpha(\nabla u)\cdot(\nabla w) \;\dd\vx.
\end{align*}
In the next subsection, we discuss a discretisation of~\eqref{eq:WF2} with classical finite elements.

\subsection{Classical finite elements}

Let $\cT_H$ denote a regular quasi-uniform simplicial/square/cubic mesh of $\Omega$ of mesh width $H$. In particular, we assume that
\begin{align*}
\max_{e\in\cT_H} H_e = H, \quad \frac{\max_{e\in\cT_H} H_e}{\min_{e\in\cT_H} H_e} \leq \kappa, \quad \max_{e\in\cT_H} \frac{H_e}{\rho_e} \leq \gamma,
\end{align*}
for some positive constants $\kappa$ and $\gamma$, where $H_e$ denotes the diameter of the element $e$ and $\rho_e$ denotes the diameter of the largest sphere that can be contained in $e$.
Furthermore, let $U_H\subset H_0^1(\Omega)$ denote the corresponding linear finite element space, i.e.
\begin{align*}
U_H &:= \{ u\in H_0^1(\Omega) \;:\; u|_e \in U_e \text{ for all } e\in\cT_H  \},
\end{align*}
where $U_e$ denotes the element space. For simplicial elements, we set $U_e=P_1(e)$, where $P_k(e)$ denotes the space of all polynomials on $e$ of degree $k$ or less, and for square/cubic meshes, we set $U_e=Q_1(e)$, where $Q_k(e)$ denotes the space of polynomials in $\mathrm{span}\{x_1^{i_1}x_2^{i_2}\cdots x_d^{i_d}\;|\; i_1,i_2,\dots,i_d\leq k\}$ on $e$. The classical finite element method can be formulated as finding $u_H\colon[0,T]\rightarrow U_H$ such that $u_H|_{t=0}=\partial_tu_H|_{t=0}\equiv 0$, and such that
\begin{align}
\label{eq:FEM}
b(\partial_t^2u_H(\cdot,t), w) + a( u_H(\cdot,t), w) = b(f(\cdot,t),w) \qquad\text{for all $w\in U_H$ and $t\in(0,T)$.}
\end{align}
Equation~\eqref{eq:FEM} can be rewritten as a system of ordinary differential equations. Let $\cX_H$ denote the set of all vertices of $\cT_H$ and let $\{\vx_i\}_{i=1}^{N_x}=\cX_H\setminus\partial\Omega$ be an ordering of all the vertices in the interior of $\Omega$. Also, let $\phi_i\in U_H$ denote the nodal basis function corresponding to the node $\vx_i$, i.e., 
\begin{align}
\label{eq:nodalBF}
\phi_i(\vx_j) &= \begin{cases}
1, &i=j, \\
0, &i\neq j.
\end{cases}
\end{align}
Next, define matrices $M,A\in\mR^{N_x\times N_x}$ and vector functions $\vvu, \vvf^*\colon[0,T]\rightarrow\mR^{N_x}$ as follows:
\begin{align*}
M_{ij} &:= b(\phi_i,\phi_j), \qquad\quad A_{ij} := a(\phi_i,\phi_j), \\
\vvu_{i}(t) &:= u_H(\vx_i,t), \qquad \vvf^*_i(t) := b(f(\cdot,t),\phi_i).
\end{align*}
Equation \eqref{eq:FEM} can then be rewritten as the following system of ordinary differential equations:
\begin{align}
\label{eq:ODE}
M\partial_t^2\vvu(t) + A\vvu(t) = \vvf^*(t) \qquad\text{for all } t\in(0,T).
\end{align}
For the time discretisation, the term $\partial_t^2\vvu(t)$ can be approximated using a central difference scheme. Let $N_t$ denote the number of time steps, let $\Delta t = T/N_t$ be the time step size, $t^n:=n\Delta t$, and let $\vvu^n$ represent $\vvu(t^n)$. We can approximate $\partial_t^2\vvu(t^n)$ by the central difference $(\vvu^{n+1}-2\vvu^n+\vvu^{n-1})/\Delta t^2$. Equation \eqref{eq:ODE} then becomes
\begin{align*}
M \frac{\vvu^{n+1}-2\vvu^n+\vvu^{n-1}}{\Delta t^2} + A\vvu^n = \vvf^*(t^n)  \qquad\text{for any $n \in \mathbb{N}$ with $0<n<N_t$.}
\end{align*}
This scheme is also known as the \emph{leapfrog} scheme. 
Rewriting this equation, we obtain the following time-stepping scheme,
\begin{align}
\label{eq:ODE2}
\vvu^{n+1} &= -\vvu^{n-1} + 2\vvu^n  + \Delta t^2 (-M^{-1}A\vvu^n + \vvf(t^n)) \qquad\text{for $n \in \mathbb{N}$, $0<n<N_t$,}
\end{align}
where $\vvf(t):=M^{-1}\vvf^*(t)$. The fully discrete finite element method can be formulated as finding $\vvu^n\in\mR^{N_x}$ for $n=0,1,\dots,N_t$, such that $\vvu^0=\vvu^1=0$, and such that \eqref{eq:ODE2} is satisfied.

Note that, in order to compute $\vvu^{n+1}$ using \eqref{eq:ODE2}, one needs to solve a linear system of the form $M\vvx=\vvy$. For large meshes, solving this linear system of equations is computationally inefficient unless $M$ is a diagonal matrix. To circumvent this problem, one can use a technique known as mass lumping, which will be discussed in the next subsection.

\subsection{Mass lumping}

To obtain a suitable diagonal matrix $M$, we can replace the bilinear operator $b(\cdot,\cdot)$ by a discrete variant $b_H(\cdot,\cdot)$, defined by
\begin{align}
\label{eq:MLproduct}
b_H(u,w) := \int_{\Omega} I_H(\beta u w) \;\dd\vx,
\end{align}  
where $I_H\colon\cC^0(\overline\Omega)\rightarrow U_H$ denotes the interpolation operator that interpolates a continuous function sampled at the vertices $\cX_H$ by a function in $U_H$. The vector function $\vvf$ and the matrix $M$ can then be written as
\begin{align*}
\qquad \vvf_i(t) = f(\vx_i,t),\qquad M_{ij} = 
\begin{cases}
\beta(\vx_i) \int_{\Omega} \phi_i\;\dd\vx,& i=j, \\
0, & i\neq j. 
\end{cases}
\end{align*}
The matrix $M$ is commonly known as the mass matrix and the technique of obtaining a diagonal mass matrix is called mass lumping.

We can rewrite the mass lumped finite element method as follows: find $u_H^n\in U_H$, for $n=0,1,\dots,N_t$, such that $u_H^0=u_H^1\equiv 0$ and such that
\begin{align*}
u_H^{n+1} = -u_H^{n-1} + 2u_H^n  + \Delta t^2 (-\cL_H u_H^n + I_Hf(\cdot,t^n)) \qquad  \text{for } n=1,2,\dots,N_t-1,
\end{align*}
where $u_H^n$ represents $u_H(\cdot,t^n)$ and where $\cL_H\colon U_H\rightarrow U_H$ is defined by
\begin{align*}
(\cL_H u)(\vx_i) = \frac{a( u, \phi_i)}{\beta(\vx_i) \int_{\Omega} \phi_i \;\dd \vx} \qquad\text{for all } i\leq N_x.
\end{align*}

It is well-known that -- in case of sufficiently smooth $\partial\Omega$, $\alpha$, $\beta$, and $f$ -- the first-order finite element method (with lumping) converges at an optimal order \cite{baker1976}, i.e., a linear convergence rate in the $L^{\infty}(0,T;H^1(\Omega))$ norm. However, if $\partial\Omega$, $\alpha$, and $\beta$ contain features of a much finer scale than the mesh width $H$, the classical error estimates are no longer valid for the $H$ values of interest and the expected convergence rate will only be observed when $H$ resolves the fine scales. To improve the accuracy of the finite element approximation, one can replace the finite element space $U_H$ by an adapted finite element space that takes into account local fine-scale features of the coefficients. Here, we employ the multiscale technique known as Localized Orthogonal Decomposition (LOD) that was presented in~\cite{MalP14} and further developed in~\cite{HenP13}. The method is introduced in the following section.

\section{Mass lumped multiscale approach for the wave equation}\label{s:multiscale}

\subsection{The Localized Orthogonal Decomposition method}

The construction of the LOD is heavily based upon a local quasi-interpolation operator $\Pi_H\colon H^1_0(\Omega) \to U_H$ which fulfils the following stability and interpolation properties,
\begin{subequations}
\begin{align}
\|H^{-1}\,(1 - \Pi_H)v\|_{0;e} + |\Pi_H v|_{1;e} &\leq C_\mathrm{I}\, |v|_{1;\Nb(e)}, && v \in H^1_0(\Omega),\,\; e \in \cT_H,\label{eq:IH1loc}\\ 
\|\Pi_H v\|_{0;e} &\leq C_\mathrm{I}\, \|v\|_{0;\Nb(e)}, && v \in L^2(\Omega),\;\, e \in \cT_H, \label{eq:IH2loc}\\
\Pi_H \circ \Pi_H &= \Pi_H, \label{eq:IH3}
\end{align}
\end{subequations}
where $\Nb(e)$ denotes the neighbourhood of an element $e \in \cT_H$ defined by
\begin{equation*}
\Nb(e) := \bigcup\bigl\{\hat{e}\in\cT_H\,\colon\,\overline{\hat{e}}\cap\overline{e} \neq\emptyset\bigr\}.
\end{equation*} 
For later use, we also define the \emph{neighbourhood of order $\ell \in \mathbb{N}$} (or \emph{$\ell$-patch}) of $e \in \cT_H$ by
\begin{equation*}
\Nb^\ell(e)=\Nb(\Nb^{\ell-1}(e)),\quad \ell \geq 2,\quad \text{ and }\quad\Nb^1(e)=\Nb(e).
\end{equation*}
Going back to the operator $\Pi_H$, we can now define a decomposition of a function $v \in H^1_0(\Omega)$ into a coarse-scale part $\Pi_Hv \in U_H$ and a fine-scale part $(1-\Pi_H)v \in W_H := \ker \Pi_H$. The idea of the method is to enhance any function $v_H \in U_H$ by adding an appropriate function in $W_H$. 
This is done in a localized way, i.e., for a fixed localization parameter $\ell \in \mathbb{N}$ and an element $e \in \cT_H$, we define for a function $v \in H^1_0(\Omega)$ its \emph{local correction} $\cQ^\ell_ev \in W_H(\Nb^\ell(e))$ by
\begin{equation}\label{eq:loccor}
a(\cQ^\ell_ev,w) = \int_e \alpha \nabla v \cdot \nabla w \dx \qquad\text{for all $w \in W_H(\Nb^\ell(e))$,}
\end{equation} 
where 
\begin{equation*} 
W_H(\Nb^\ell(e)):=\{w\in W_H \,: \,w\vert_{\Omega \setminus \Nb^\ell(e)}=0 \}
\end{equation*}
denotes the restriction of the kernel space $W_H$ to $\Nb^\ell(e)$. 
The \emph{global correction} $\cQ^\ell\colon H^1_0(\Omega) \to W_H$ is finally defined by 
\begin{equation*}\label{eq:globcor}
\cQ^\ell v := \sum_{e \in \cT_H} \cQ^\ell_e v.
\end{equation*}
We emphasize that the operator $\cQ^\ell$ fulfils the stability property $\|\cQ^\ell v\|_1 \leq C \,\| v\|_1$ for all $v \in H^1_0(\Omega)$.

We can now define a new (multiscale) test and ansatz space $\tU^\ell_H := \cS^\ell_H(U_H)$, with the multiscale operator $\cS_H^\ell := 1 - \cQ^\ell$. Note that a basis of $\tU^\ell_H$ can be obtained from classical nodal basis functions via $\cS_H^\ell\phi_i$, $i = 1, \ldots N_x$. In particular, $\dim (\tU^\ell_H) = \dim (U_H)$.

The localization parameter $\ell$ is crucial to avoid solving global problems when constructing the multiscale operator. It is, however, important to appropriately couple $\ell$ to the mesh parameter $H$ as discussed in Section~\ref{ss:locerr} below. Nevertheless, non-localized corrections (i.e., $\ell = \infty$) are an important tool for the theoretical error analysis. In particular, the following orthogonality property holds naturally in the non-localized setting,
\begin{equation}\label{eq:ortho}
a(\cS_H v, w) = 0 \qquad\text{for any $v \in H^1_0(\Omega)$ and $w \in W_H$, }
\end{equation}
where here and in the following we abbreviate $\cS_H := \cS_H^\infty$.
In other words, $\tU_H:=\cS_H(U_H)$ and $W_H$ form an orthogonal decomposition of $H_0^1(\Omega)=\tU_H \oplus W_H$ with respect to the inner product $a(\cdot, \cdot)$. This orthogonality gives the method its name. 

The semi-discrete formulation with the LOD multiscale space for the spatial discretisation can be written as finding $\tu^\ell_H\colon[0,T]\rightarrow \tU^\ell_H$ such that $\tu^\ell_H|_{t=0} =\partial_t\tu^\ell_H|_{t=0} \equiv 0$ and
\begin{align}
\label{eq:LOD}
b(\partial_t^2\tu^\ell_H(\cdot,t), w) + a(\tu^\ell_H(\cdot,t), w) = b(f(\cdot,t),w) \qquad\text{for all $w\in \tU^\ell_H$ and $t\in(0,T)$. }
\end{align}
It follows from~\cite[Theorem~4.7]{AbdH17} that, if $\beta \equiv 1$, $f$ is sufficiently smooth, and $\ell \geq C|\log H|$, the error between the solution $u$ of~\eqref{eq:WF} and the solution $\tu_H^\ell$ of~\eqref{eq:LOD} can be estimated by
\begin{align}
\label{eq:errBoundLOD}
 \| u-\tu^\ell_H \|_{\infty,1} \leq C\,H^2,
\end{align}
where $C$ in~\eqref{eq:errBoundLOD} depends on $\alpha_{\min},\,\alpha_{\max},\,\Omega,\,d$, and norms of $f$, but not on the fine scale on which the coefficient $\alpha$ varies. 
Such a bound also holds for the case $\beta \not\equiv 1$ if a specific $\beta$-dependent operator $\Pi_H$ is used, as investigated in Section~\ref{s:error}. 

For mass lumping, one can no longer use the discrete inner product $b_H(\cdot,\cdot)$ as defined in \eqref{eq:MLproduct}, since $\beta$ is not necessarily continuous in the very general setting and therefore the integrand in \eqref{eq:MLproduct} might not be well-defined. Further, the adapted basis functions $\cS^\ell_H\phi_i$ may no longer satisfy the property of nodal basis functions given in \eqref{eq:nodalBF}. To apply mass lumping to the LOD method and maintain an error bound of the form \eqref{eq:errBoundLOD}, one needs to carefully choose the discrete inner product $b_H(\cdot,\cdot)$ and the projection operator $\Pi_H$. A suitable mass lumping technique will be presented and analysed in the next subsection.

\subsection{The mass lumped LOD method}

For $k \in \mathbb{N}_0$, let $U_{H,k}^{D}$ denote the \emph{discontinuous} piece-wise polynomial function space given by
\begin{align*}
U^D_{H,k} &:= \{ u\in L^2(\Omega) \;:\; u|_e \in U_{e,k} \,\forall e\in\cT_H\},
\end{align*}
where $U_{e,k}=P_k(e)$ for simplicial elements, and $U_{e,k}=Q_k(e)$ for square/cubic elements. Observe that $U_H=U_{H,1}^D\cap H_0^1(\Omega)$. Also, let $\Pi^D_{H,k}\colon L^2(\Omega)\rightarrow U_{H,k}^D$ denote the weighted $L^2(\Omega)$ projection operator defined by
\begin{align*}
b(\Pi_{H,k}^Du,w) = b(u,w) \qquad\text{for all } w\in U^D_{H,k}.
\end{align*}
Furthermore, for each element $e\in\cT_H$, let $\cX_e$ denote the set of vertices of $e$ and let $I_e\colon\cC^0(\overline e)\rightarrow U_e$ denote the interpolation operator that interpolates a function sampled at the vertices of $e$ by a polynomial in $U_e$. We define the following discrete variant of the bilinear operator $b$:
\begin{align*}
b_H'(u, w) := \sum_{e\in\cT_H} \int_e \beta I_e(uw)\;\dd\vx, \qquad u,w\in U_{H,1}^D.
\end{align*}
Note that
\begin{align*}
b_H'(\phi_i,\phi_j) = \begin{cases}
\int_\Omega \beta\phi_i \;\dd\vx, &i=j, \\
0, &\text{otherwise}.
\end{cases}
\end{align*}
We also define the averaging operator $\Pi^\mathrm{av}_{H}\colon U_{H,1}^D\rightarrow U_H$ by the property that
\begin{align*}
b_H'(\Pi_H^\mathrm{av}u, w) = b_H'(u,w) \qquad\text{ for all } w\in U_H.
\end{align*}
This is equivalent to writing
\begin{align}
\label{eq:defPiHav2}
(\Pi_H^\mathrm{av}u)(\vx_i) = \frac{ \sum_{e:\vx_i \in\overline{e}} (\int_{e} \beta\phi_i \;\dd\vx)u_e(\vx_i) }{ \int_{\Omega} \beta\phi_i \;\dd\vx} \qquad\text{ for all } i\leq N_x,
\end{align}
where $u_e:= u|_e$. In other words, $\Pi^\mathrm{av}_H u$ is obtained by taking a weighted average of the values at the vertices of a discontinuous function $u\in U_{H,1}^D$.

Next, we define the projection operator $\Pi_H\colon L^2(\Omega)\rightarrow U_H$ as $\Pi_H:=\Pi_H^\mathrm{av}\circ\Pi_{H,1}^D$. 
Further, we introduce the semi-discrete mass lumped LOD method that seeks $\tu_H^\ell := \cS_H^{\ell} u^\ell_H\colon [0,T]\rightarrow \tU^\ell_H$, where 
$u^\ell_H\colon[0,T]\rightarrow U_H$ is such that $u^\ell_H|_{t=0}=\partial_tu^\ell_H|_{t=0} \equiv 0$ and 
\begin{align}
\label{eq:MLLOD2}
b_H'(\partial_t^2 u^\ell_H(\cdot,t), w) + a(\cS^\ell_Hu^\ell_H(\cdot,t), \cS^\ell_H w) &= b_H'(\Pi_Hf(\cdot,t) ,w) \qquad\text{for all $w\in U_H$, $t\in(0,T)$.}
\end{align}
With the equality $\Pi_H\cS^\ell_H w = w$ for all $w \in U_H$, 
we can reformulate~\eqref{eq:MLLOD2} as 
\begin{align}
\label{eq:MLLOD}
b_H'(\Pi_H \partial_t^2\tu^\ell_H(\cdot,t), \Pi_Hw) + a(\tu^\ell_H(\cdot,t),w) &= b_H'(\Pi_Hf(\cdot,t),\Pi_Hw) \quad\text{for all $w\in \tU^\ell_H$, $t\in(0,T)$.}
\end{align} 
The corresponding fully discrete mass lumped LOD method can be stated as finding $\tu_H^{\ell,n}\in \tU_H^{\ell}$ for $n=0,1,\dots,N_t$ such that $\tu_H^{\ell,0}=\tu_H^{\ell,1}\equiv 0$ and
\begin{align}
\label{eq:MLLODfull}
b_H'\left(\Pi_H \frac{\tu^{\ell,n+1}_H - 2\tu^{\ell,n}_H + \tu^{\ell,n-1}_H}{\Delta t^2}, \Pi_Hw \right) + a(\tu^{\ell,n}_H,w) &= b_H'(\Pi_Hf(\cdot,t^n),\Pi_Hw) 
\end{align}
for all $w\in \tU_H^{\ell}$ and $n=1,2,\dots,N_t-1$. Here, $\tu_H^{\ell,n}$ approximates $\tu^\ell_H(\cdot,t^n)$.
As before, this is equivalent to writing $\tu_H^{\ell,n} = \cS_H^{\ell} u_H^{\ell,n}$, where $u_H^{\ell,n}\in U_H$ is defined such that $u_H^{\ell,0}=u_H^{\ell,1}\equiv 0$ and 
\begin{align}
\label{eq:MLLODfull2}
b_H'\left(\frac{u^{\ell,n+1}_H - 2u^{\ell,n}_H + u^{\ell,n-1}_H}{\Delta t^2}, w \right) + a(\cS_H^{\ell} u^{\ell,n}_H, \cS_H^{\ell} w) &= b_H'(\Pi_Hf(\cdot,t^n),w) 
\end{align}
for all $w\in U_H^{\ell}$ and $n=1,2,\dots,N_t-1$.
We can also rewrite \eqref{eq:MLLODfull2} as
\begin{align}
\label{eq:MLLODfull3}
u_H^{\ell,n+1} &= -u_H^{\ell,n-1} + 2u_H^{\ell,n}  + \Delta t^2 (-\cL^\ell_{H} u_H^{\ell,n} + \Pi_Hf(\cdot,t^n)),
\end{align}
where $\cL_{H}^\ell\colon U_H\rightarrow U_H$ is defined by
\begin{align*}
(\cL_H^\ell u)(\vx_i) = \frac{a(\cS^\ell_Hu,\cS^\ell_H\phi_i)}{\int_\Omega \beta\phi_i \;\dd\vx} \qquad\text{for all $i\leq N_x$.}
\end{align*}
A detailed error analysis for the method presented in~\eqref{eq:MLLODfull} is given in the following section. 

\begin{rem}[Right-hand side]
Note that we have $\Pi_H f$ in~\eqref{eq:MLLOD2}--\eqref{eq:MLLODfull3} to ensure the necessary continuity for the operator $I_e$ in the definition of $b_H'$. If $f$ is regular enough in space, e.g., $H^2$-regular if $d \leq 3$, one can instead use $f$ directly without an influence on the convergence rate. 
\end{rem}

\begin{rem}[Correct averaging operator]
The averaging strategy used in \eqref{eq:defPiHav2} averages the values at a given node based on a $\beta$-dependent weighting. This averaging -- and therefore the explicit definition of the interpolation operator $\Pi_H$ -- is essential in order to obtain the striven second-order convergence rate. In particular, the `classical' definition of the operator $\Pi_H$ that is based on equal weighting (see \cite{ErnG17}) and that is commonly used in the context of the LOD may lead to a sub-optimal convergence rate as also illustrated in Section~\ref{s:numerics}.
\end{rem}

\section{Error analysis}\label{s:error}

This section is devoted to a complete error analysis for the mass lumped LOD method as given in \eqref{eq:MLLODfull}. First, we state the main result as a theorem below. The proof is given in Section \ref{ss:proofthm} and relies on the preliminary results and error estimates obtained in Sections \ref{ss:preresults}--\ref{ss:temperr}.

\begin{thm}[Error of the mass lumped LOD method]\label{thm:fullerr}
Let $u$ be the exact solution of \eqref{eq:WF2} and let $\tu_h^{\ell,n}$, with $\ell\in\mathbb{N}$, be the mass lumped LOD approximation given by \eqref{eq:MLLODfull}. Assume that $\partial_t^k f \in L^{\infty}(0,T;L^2(\Omega))$ for $k\leq 3$, $\partial_t^{k}f \in L^{\infty}(0,T;H^1_0(\Omega))$ for $k\leq1$, and $\partial_t^{k}f|_{t=0}\equiv 0$ for $k\leq 2$.
Further, assume that the CFL condition
\begin{equation}\label{eq:CFL}
1 - \frac12 (1 + C_\mathrm{loc}^2\,\ell^{d}\exp(-c_\mathrm{dec}2\ell))\,\frac{\alpha_{\max}}{\beta_{\min}} C_b^2 C^2_\mathrm{inv}H^{-2}\Delta t^2 \geq \delta
\end{equation}
holds true for some $\delta > 0$, with $C_b$ the constant in Lemma~\ref{lem:bH}, $C_\mathrm{loc},\,c_\mathrm{dec}$ the constants in Lemma~\ref{lem:locerrEll}, and $C_\mathrm{inv}$ the constant in~\eqref{eq:invIneq}. 
Then, we have the following error estimate: 
\begin{equation}\label{eq:finalerr}
\begin{aligned}
\max_{1\leq n\leq N_t}\|\Pi_H D_{\Delta t}u(\cdot,t^n) - \Pi_H D_{\Delta t}\tu_{H}^{\ell,n}\|_0 + & \max_{0\leq n\leq N_t} \|u(\cdot,t^n) - \tu_{H}^{\ell,n}\|_1
\leq \\
&\qquad C\,(H^2 + \Delta t^2 + \ell^{d/2}\exp(-c_\mathrm{dec}\ell)) \cdot \texttt{\upshape data},
\end{aligned}
\end{equation}
with the notation $D_{\Delta t} v_H(
\cdot,t^n) = \Delta t^{-1}(v_H(\cdot,t^n) - v_H(\cdot,t^{n-1}))$ and $D_{\Delta t}v_H^n = {\Delta t}^{-1}(v_H^n-v_H^{n-1})$, and where 
\begin{align} \label{eq:data}
\texttt{\upshape data} =  \| f \|_{\infty,1} + T \| \partial_t f \|_{\infty,1} + T \| \partial_t^2 f \|_{\infty,0} + T^2 \| \partial_t^3 f \|_{\infty,0} + T\| f \|_{\infty,0} + T^2 \| \partial_t f \|_{\infty,0}.
\end{align}
Note that the constant $C$ in~\eqref{eq:finalerr} depends on $\alpha_{\min},\,\alpha_{\max},\,\beta_{\min},\,\beta_{\max}$, and in particular on the contrast $\max\{\alpha_{\max},\beta_{\max}\}/\min\{\alpha_{\min},\beta_{\min}\}$. 
\end{thm}

\begin{rem}
For moderate choices of $\ell$, the CFL condition~\eqref{eq:CFL} basically reads
\begin{equation*}
\Delta t \leq C\sqrt{{\beta_{\min}}/{\alpha_{\max}}}\, H.
\end{equation*}
\end{rem}

\begin{rem}
\label{rem:scalel}
If $\ell$ is of the form $\ell=C |\log(H)|$, with a sufficiently large constant $C$, then the error estimate in Theorem~\ref{thm:fullerr} is of the form
\begin{align*}
\max_{1\leq n\leq N_t}\|\Pi_H D_{\Delta t}u(\cdot,t^n) - \Pi_H D_{\Delta t}\tu_{H}^{\ell,n}\|_0 + & \max_{0\leq n\leq N_t} \|u(\cdot,t^n) - \tu_{H}^{\ell,n}\|_1
\leq C (H^2 + \Delta t^2)\cdot \texttt{\upshape data}.
\end{align*}
\end{rem}

\begin{rem} \label{rem:fnorms}
If $\ell$ is of the form as in Remark \ref{rem:scalel} and if the norms of $f$ in \eqref{eq:data} are bounded independently of the scales at which $\alpha$ and $\beta$ vary, i.e., 
\begin{align*}
\sum_{i=0}^3 \|\partial_t^i f\|_{\infty,0} + \sum_{i=0}^1 \|\partial_t^i f\|_{\infty,1} &\leq C,
\end{align*}
then the error estimate in Theorem \ref{thm:fullerr} is of the form
\begin{align*}
\max_{1\leq n\leq N_t}\|\Pi_H D_{\Delta t}u(\cdot,t^n) - \Pi_H D_{\Delta t}\tu_{H}^{\ell,n}\|_0 + & \max_{0\leq n\leq N_t} \|u(\cdot,t^n) - \tu_{H}^{\ell,n}\|_1
\leq C (H^2+\Delta t^2).
\end{align*}
\end{rem}

\begin{rem}[Initial conditions]\label{rem:initial}
In this work, we restrict ourselves to zero initial conditions in order to improve readability. However, it is important to mention that non-zero initial conditions can be considered as well if appropriate regularity assumptions are met; see, e.g., \cite{AbdH17, MaiP19}. In this case, the approximation at time $t^1$ needs to be chosen more carefully to retain the above convergence rates. Finally, we note that for non-zero initial conditions, one can relax the zero initial conditions on $f$.
\end{rem}

\begin{rem}[Boundary condition on $f$]\label{rem:boundary}
To obtain a second-order convergence rate, we require that $f$ is zero on the boundary. If this is not the case, it can still be shown that the convergence rate is of first order. This was for instance investigated and observed in~\cite{MaiP19} for the particular case of $\beta \equiv 1$ without mass lumping.
\end{rem}
As mentioned in Remark \ref{rem:scalel}, the choice $\ell \approx |\log H|$ is enough to obtain a second-order convergence rate in $H$ and $\Delta t$. Practical experiments, however, indicate that a fixed parameter $\ell$ is often already enough and an adjustment with $H$ is not always required.

Before we turn to proving the main error estimates, we state and prove several preliminary results in the following subsection.

\subsection{Preliminary results}
\label{ss:preresults}

\begin{lem}
\label{lem:a}
The bilinear operator $a$ is an inner product on $H_0^1(\Omega)$ and the corresponding norm $\|u\|_{a}:=\sqrt{a(u,u)}$ is equivalent to the $H^1(\Omega)$ norm in the sense that
\begin{align*}
\alpha_{\min}^{1/2}C_a^{-1}\|u\|_1 \leq \|u\|_{a} \leq \alpha_{\max}^{1/2}C_a\|u\|_1 \qquad\text{for all } u\in H_0^1(\Omega).
\end{align*}
for some positive constant $C_a=C_a(\Omega)$ that does not depend on the mesh or the coefficient $\alpha$, but does depend on the diameter of $\Omega$.
\end{lem}

\begin{proof}
The result follows immediately from the definition of $a$, the fact that $\alpha_{\min}\leq \alpha \leq \alpha_{\max}$, and the Poincar\'e inequality.
\end{proof}

\begin{lem}
\label{lem:b}
The bilinear operator $b$ is an inner product on $L^2(\Omega)$ and the corresponding norm $\|u\|_{b}:=(b(u,u))^{1/2}$ is equivalent to the $L^2(\Omega)$ norm in the sense that
\begin{align*}
\beta_{\min}^{1/2}\|u\|_0 \leq \|u\|_{b} \leq \beta_{\max}^{1/2} \|u\|_0 \qquad\text{for all } u\in L^2(\Omega).
\end{align*}
\end{lem}

\begin{proof}
The result follows immediately from the definition of $b$.
\end{proof}

\begin{lem}
\label{lem:bH}
The bilinear operator $b_H'$ acts as an inner product on $U_{H,1}^{D}$ and the corresponding norm $\|u\|_{b_H'}:=(b_H'(u,u))^{1/2}$ is equivalent to the $L^2(\Omega)$ norm in the sense that
\begin{align*}
\beta_{\min}^{1/2}C_b^{-1}\|u\|_0 \leq \|u\|_{b_H'} \leq \beta_{\max}^{1/2}C_b\|u\|_0 \qquad\text{for all } u\in U_{H,1}^D,
\end{align*}
for some positive constant $C_b$ that does not depend on the mesh or the coefficient $\beta$.
\end{lem}

\begin{proof}
Let $\he$ be a reference element and define the bilinear operator $\hb(u,w):=\int_{\he} I_{\he}(uw)\;\dd\vx$. Since $\hb$ defines an inner product on the reference space $U_{\he}$, we have that $\|u\|_{\hb}:=(\hb(u,u))^{1/2}$ is a norm on $U_{\he}$. Since $U_{\he}$ is finite dimensional, and since all norms on finite dimensional spaces are equivalent, we have
\begin{equation*}
C_{b}^{-1}\|u\|_{0,\he} \leq \|u\|_{\hb} \leq C_{b}\|u\|_{0,\he} \qquad\text{for all } u\in U_{\he}
\end{equation*}
for some positive constant $C_{b}$, where $\|\cdot\|_{0,\he}$ denotes the $L^2(\he)$ norm. The Lemma then follows from a standard mapping argument.
\end{proof}

\begin{lem}
\label{lem:stab}
The projection operators $\Pi_{H,k}^D$ and $\Pi_H$ are stable in $L^2(\Omega)$ in the following sense:
\begin{subequations}
\begin{align}
\|\Pi_{H,k}^Du\|_{0;e} &\leq C\|u\|_{0;e} &&\text{for all } u\in L^2(e),\, e\in\cT_H,\, k=0,1, \label{eq:stabPiHD} \\
\|\Pi_H u\|_{0;e} &\leq C\|u\|_{0;N(e)} &&\text{for all } u\in L^2(e),\, e\in\cT_H. \label{eq:stabPiH}
\end{align}
\end{subequations}
\end{lem}

\begin{proof}
From Lemma \ref{lem:b} and the fact that $\Pi_{H,k}^D$ is an orthogonal projection operator with respect to the inner product $b$, it follows that $\|\Pi_{H,k}^D u\|_{0}\leq C\|u\|_0$ for all $u\in L^2(\Omega)$. The first inequality then follows from the fact that $(\Pi_{H,k}^D u)|_e = (\Pi_{H,k}^D (\chi_e u))|_e$, where $\chi_S$ denotes the characteristic function that equals $1$ in the subdomain $S \subseteq \Omega$ and $0$ elsewhere. 

For the second inequality, we first observe that $\Pi_H^\mathrm{av}$ is an orthogonal projection operator with respect to the inner product $b_H'$. From Lemma \ref{lem:bH}, it then follows that $\|\Pi^\mathrm{av}_Hu\|_0 \leq C\|u\|_0$ for all $u\in U_{H,1}^D$.
Using \eqref{eq:stabPiHD} and the definition of $\Pi_H$, we can then derive
\begin{align}
\label{eq:stabPiHD1a}
\|\Pi_Hu\|_0 &= \|\Pi_H^\mathrm{av}(\Pi_{H,1}^Du)\|_0 \leq C\|\Pi_{H,1}^Du\|_0 \leq C\|u\|_0 \qquad\text{for all } u\in L^2(\Omega).
\end{align}
From \eqref{eq:defPiHav2}, it follows that $(\Pi_{H}u)|_e = (\Pi_H(\chi_{N(e)}u))|_e$. Inequality \eqref{eq:stabPiH} then follows from \eqref{eq:stabPiHD1a}.
\end{proof}

\begin{lem}
\label{lem:interp}
We have the interpolation properties
\begin{subequations}
\begin{align}
\| u-\Pi_{H,k}^Du \|_{0;e} &\leq CH|u|_{1;e} &&\text{for all } u\in H_0^1(\Omega),\, e\in\cT_H,\, k=0,1, \label{eq:intPiHD} \\
\| u-\Pi_Hu \|_{0;e} &\leq CH|u|_{1;N(e)} &&\text{for all } u\in H_0^1(\Omega),\, e\in\cT_H. \label{eq:intPiH}
\end{align}
\end{subequations}
\end{lem}

\begin{proof}
The projection operator $\Pi_{H,k}^D$ is stable in $L^2(\Omega)$ (see Lemma \ref{lem:stab}) and preserves piece-wise constant functions. Inequality \eqref{eq:intPiHD} then follows from standard interpolation theory, see, e.g., \cite[Chapter 3.1]{Cia78}.

Now, let $\overline{u}_{N(e)}$ denote the average of $u$ on the domain $N(e)$. 
Note that $\Pi_H$ preserves functions in $U_H$, i.e., $\Pi_H w_H = w_H$ for all $w_H\in u_H$. Using the triangle inequality, Lemma \ref{lem:stab}, and the Poincar\'e inequality, we can then obtain, for any $w_H\in U_H$,
\begin{align*}
\|u-\Pi_Hu\|_{0;e} &= \|u-w_H - \Pi_H(u-w_H)\|_{0;e} \\
&\leq \|u-w_H\|_{0;e} + \|\Pi_H(u-w_H)\|_{0;e} \\
&\leq C \|u-w_H\|_{0;N(e)}.
\end{align*}
We can choose $w_H\equiv 0$ when $N(e)$ is attached to the boundary $\partial\Omega$ and choose any function $w_H \in U_H$ with $w_H\vert_{N(e)} = \overline{u}_{N(e)}$ otherwise. Inequality \eqref{eq:intPiH} then follows from the above and the Poincar\'e inequality.
\end{proof}

\begin{lem}
\label{lem:stab2}
The projection operator $\Pi_H$ is stable in $H^1_0(\Omega)$ in the following sense:
\begin{align*}
|\Pi_H u|_{1;e} &\leq C|u|_{1;N(e)} \qquad\text{for all } u\in H_0^1(e),\, e\in\cT_H. \label{eq:stabPiH2}
\end{align*}
\end{lem}

\begin{proof}
Let $\Pi^*_H$ denote the Cl\'ement quasi-interpolation operator, see \cite{Cle75}. This quasi-interpolation operator satisfies
\begin{align*}
\| u-\Pi_H^*u \|_{0;e} &\leq CH|u|_{1;N(e)} &&\text{for all } u\in H_0^1(\Omega),\, e\in\cT_H, \\
|\Pi_H^*u|_{1;e} &\leq C|u|_{1;N(e)} &&\text{for all } u\in H_0^1(\Omega),\, e\in\cT_H.
\end{align*}
Using the triangle inequality, the inverse inequality, and Lemma \ref{lem:interp}, we then obtain
\begin{align*}
|\Pi_H u|_{1;e} &\leq |(\Pi_Hu - \Pi_H^*u)|_{1;e} + |\Pi_H^*u|_{1;e} \\
&\leq CH^{-1}\| \Pi_Hu - \Pi_H^*u \|_{0;e} + |\Pi_H^*u|_{1;e} \\
&\leq CH^{-1}(\| u - \Pi_Hu \|_{0;e} + \| u - \Pi_H^*u \|_{0;e}) + |\Pi_H^*u|_{1;e} \\
&\leq C|u|_{1;N(e)}. \qedhere
\end{align*}
\end{proof}

From these results, it follows that $\Pi_H$ satisfies conditions \eqref{eq:IH1loc}--\eqref{eq:IH3}. In particular, \eqref{eq:IH1loc} follows from Lemma \ref{lem:interp} and Lemma \ref{lem:stab2}, \eqref{eq:IH2loc} follows from Lemma \ref{lem:stab}, and \eqref{eq:IH3} follows immediately from the definition of $\Pi_H$.

\begin{lem}[Regularity of $u$]\label{lem:regEst}
Let $u$ denote the solution to \eqref{eq:WF}. Also, let $k\geq 1$ and 
\begin{subequations}
\label{eq:regcond}
\begin{align}
\partial_t^i f \in L^{\infty}(0,T;L^2(\Omega)) &\qquad\text{for }i\leq k, \\
\partial_t^if|_{t=0} \equiv 0 &\qquad\text{for }i\leq k-1.\label{eq:regcond_f}
\end{align}
\end{subequations}
Then $\partial_t^i u\in L^{\infty}(0,T;H^1_0(\Omega))$ for $i\leq k$, $\partial_t^{k+1}u \in L^{\infty}(0,T;L^2(\Omega))$, and $\partial_t^i u|_{t=0}\equiv 0$ for $i\leq k+1$. Furthermore, we have the estimates
\begin{align*}
\|\partial_t^{i+1} u\|_{\infty,0} + \|\partial_t^i u\|_{\infty,1} &\leq CT\|\partial_t^i f\|_{\infty,0} \qquad\text{for }i\leq k.
\end{align*}
\end{lem}

\begin{proof}
The result follows from the proof of \cite[Theorem 8.1, Remark 8.2]{LioM72} by replacing the inner product $(\cdot,\cdot)$ by the weighted $L^2(\Omega)$ inner product $(\beta\cdot,\cdot)$, by taking higher-order time derivatives of the equations for the discrete approximation, and by using the Cauchy--Schwarz inequality instead of Gronwall's lemma for obtaining the energy estimates.
\end{proof}

\begin{lem}[Regularity of $\tu_H^\ell$]\label{lem:regEstH}
Let $\tu_H^\ell$ be the solution to~\eqref{eq:MLLOD} for $\ell \in \mathbb{N}\cup\{\infty\}$. Also, assume that $f$ satisfies \eqref{eq:regcond} for some $k\geq 1$. Then $\partial_t^i \tu_H^\ell\in L^{\infty}(0,T;H^1_0(\Omega))$ for $i\leq k$, $\partial_t^{k+1}\tu_H^\ell \in L^{\infty}(0,T;L^2(\Omega))$, and $\partial_t^i \tu_H^\ell|_{t=0}\equiv 0$ for $i\leq k+1$. Furthermore, we have the estimates
\begin{align*}
\|\Pi_H\partial_t^{i+1} \tu_H^\ell\|_{\infty,0} + \|\partial_t^i \tu_H^\ell\|_{\infty,1} &\leq CT\|\partial_t^i f\|_{\infty,0} \qquad\text{for }i\leq k.
\end{align*}
\end{lem}

\begin{proof}
The proof is analogous to that of Lemma \ref{lem:regEst}, but now we only have to take into account the discrete system and the fact that the norm $\|\cdot\|_{b_H'}$ is equivalent to $\|\cdot\|_0$ for the discrete space $U_H$ as stated in Lemma \ref{lem:b}. 
\end{proof}
	
\subsection{Discretisation error of $b$}

This subsection is devoted to showing that the discretisation error between $b$ and $b_H'$ is reasonably small. Therefore, let $r_H(u,w):= b(u,w)-b_H'(\Pi_Hu,\Pi_Hw)$. We can now derive the following bound.
\begin{lem}
\label{lem:rH}
The discretisation error of $b$ can be estimated by
\begin{align}
\label{eq:boundrH}
|r_H(u,w)| &\leq CH^2 \|u\|_1 \|w\|_1 \qquad\text{for all } u,w\in H_0^1(\Omega).
\end{align}
\end{lem}

\begin{proof}
We can write
\begin{align*}
r_H(u,w) = r_{H,1}(u,w) + r_{H,2}(\Pi_{H,1}^D u,\Pi_{H,1}^D w) + r_{H,3}(\Pi_{H,1}^D u, \Pi_{H,1}^D w) \qquad\text{for all } u,w \in H_0^1(\Omega),
\end{align*}
where
\begin{align*}
r_{H,1}(u,w)  &:= b(u,w) - b(\Pi_{H,1}^Du, \Pi_{H,1}^Dw), &&u,w\in H_0^1(\Omega),  \\
r_{H,2}(u,w)  &:= b(u, w) - b_H'(u, w), && u,w\in U_{H,1}^D, \\
r_{H,3}(u,w)  &:= b_H'(u,w) - b_H'(\Pi_{H}^\mathrm{av}u, \Pi_{H}^\mathrm{av}w), &&u,w\in U_{H,1}^D.
\end{align*}
From the definition of $\Pi_{H,1}^D$ it follows that 
\begin{align*}
b(\Pi_{H,1}^Du,w)&=b(\Pi_{H,1}^Du,\Pi_{H,1}^Dw)=b(u,\Pi_{H,1}^Dw) \qquad\text{for all } u,w\in H_0^1(\Omega).
\end{align*} 
From this, we obtain
\begin{align*}
|r_{H,1}(u,w)| &= |b(u-\Pi_{H,1}^Du,w-\Pi_{H,1}^Dw)| \\
&\leq C\|u-\Pi_{H,1}^Du\|_0 \|w-\Pi_{H,1}^Dw\|_0 \\
&\leq CH^2\|u\|_1\|w\|_1,
\end{align*}
for all $u,w\in H_0^1(\Omega)$. Here, the second line follows from Lemma \ref{lem:b} and the Cauchy--Schwarz inequality, and the last line follows from Lemma \ref{lem:interp}. 
Furthermore, we have that
\begin{align*}
r_{H,2}(u,w)&=0 \qquad\text{for all } u\in U_{H,1}^D \text{ and }w\in U_{H,0}^D, \\
r_{H,2}(u,w)&=0 \qquad\text{for all } u\in U_{H,0}^D \text{ and }w\in U_{H,1}^D.
\end{align*}
From this, we derive
\begin{align*}
| r_{H,2}(\Pi_{H,1}^D u,\Pi_{H,1}^D w) | &= | r_{H,2}(\Pi_{H,1}^D u - \Pi_{H,0}^D u, \Pi_{H,1}^D w - \Pi_{H,0}^D w) | \\
&\leq C\| \Pi_{H,1}^D u - \Pi_{H,0}^D u \|_0 \| \Pi_{H,1}^D w - \Pi_{H,0}^D w \|_0 \\
&\leq CH^2 \|u\|_1 \|w\|_1
\end{align*}
for all $u,w\in H_0^1(\Omega)$. Here, the second line follows from the triangle inequality, Lemma~\ref{lem:b}, Lemma~\ref{lem:bH}, and the Cauchy--Schwarz inequality, and the last line follows from the triangle inequality and Lemma~\ref{lem:interp}.

From the definition of $\Pi_{H}^\mathrm{av}$, we further have
\begin{align*}
b_H'(\Pi_{H}^\mathrm{av}u, w)&=b_H'(\Pi_{H}^\mathrm{av} u,\ \Pi_{H}^\mathrm{av} w)=b_H'(u, \Pi_{H}^\mathrm{av} w) \qquad\text{for all } u,w\in U_{H,1}^D.
\end{align*}
From this and the fact that $\Pi_H=\Pi_H^\mathrm{av}\circ\Pi_{H,1}^D$, we can derive
\begin{align*}
|r_{H,3}(\Pi_{H,1}^D u, \Pi_{H,1}^D w)| &= |b_{H}'(\Pi_{H,1}^D u - \Pi_Hu, \Pi_{H,1}^D w-\Pi_H w)| \\
&\leq C\| \Pi_{H,1}^D u - \Pi_Hu \|_0 \| \Pi_{H,1}^D w-\Pi_H w \|_0  \\
&\leq CH^2\|u\|_1 \|w\|_1,
\end{align*}
for all $u,w\in H_0^1(\Omega)$, where the second line follows from Lemma~\ref{lem:bH} and the Cauchy--Schwarz inequality, and the last line follows from the triangle inequality and Lemma \ref{lem:interp}.

Combining all these results yields \eqref{eq:boundrH}.
\end{proof}

\subsection{Semi-discrete error analysis}\label{ss:semidiscerr} 

In this subsection, we investigate the error between the exact solution and the semi-discrete non-localized lumped approximation~$\tu_{H}$. The result reads as follows.
\begin{lem}[Semi-discrete error]
\label{lem:errsemidisc}
Let $u$ be the solution to \eqref{eq:WF} and let $\tu_H$ be the semi-discrete solution to \eqref{eq:MLLOD} for $\ell = \infty$. Assume that $\partial_t^k u\in L^{\infty}(0,T;H^1_0(\Omega))$ for $k\leq 3$, $f\vert_{t=0}\equiv 0$, and $\partial_t^k f \in L^{\infty}(0,T;H^1_0(\Omega))$ for $k\leq 1$. Then
\begin{equation*}
\label{eq:errsemidisc}
\begin{aligned}
\|\Pi_H\partial_t(u-\tu_H)\|_{\infty,0} + \|u-&\tu_H\|_{\infty,1} 
\\&
\leq CH^2\big(  \|\partial_t^2u\|_{\infty,1} + \|f\|_{\infty,1} + T\|\partial_t^3u\|_{\infty,1}  + T\| \partial_tf \|_{\infty,1} \big).
\end{aligned}
\end{equation*}
\end{lem}

\begin{proof}
By definition of $W_H$, $\Pi_Hw\equiv 0$ for all $w\in W_H$. From the definition of $\cS_H$, it also follows that $a(\tu_H,w)=0$ for all $w\in W_H$. This implies that \eqref{eq:MLLOD} also holds for test functions $w'\in W_H$, because each individual term vanishes in that case. Since $H_0^1(\Omega)=\tU_H\oplus W_H$, we can therefore write
\begin{align}
\label{eq:MLLOD1a}
b_H'(\Pi_H \partial_t^2\tu_H(\cdot,t), \Pi_Hw) + a(\tu_H(\cdot,t),w) &= b_H'(\Pi_Hf(\cdot,t),\Pi_Hw) 
\end{align}
for all $w\in H_0^1(\Omega)$ and $t\in(0,T)$. This means that equation~\eqref{eq:MLLOD} holds true for all $w\in H_0^1(\Omega)$.

Now, due to the regularity assumption, $u$ satisfies \eqref{eq:WF2}. We subtract \eqref{eq:MLLOD1a} from \eqref{eq:WF2} and set $e_H:=u-\tu_H$. Then we obtain the following equation for the error,
\begin{align}
\label{eq:err}
b_H'(\Pi_H \partial_t^2 e_H(\cdot,t), \Pi_Hw) + a(e_H(\cdot,t), w) &= r_H\big(f(\cdot,t)-\partial_t^2u(\cdot,t),w\big)
\end{align}
for all $w\in H_0^1(\Omega)$ and almost every $t\in(0,T)$.
Note that $\partial_t \tu_H(\cdot, t) \in H^1_0(\Omega)$ by Lemma~\ref{lem:regEstH}. Therefore, we can insert $w=\partial_t e_H(\cdot, t)$ into \eqref{eq:err} and obtain
\begin{align}
\label{eq:err2}
\partial_tE_H(t) &= r_H\big(f(\cdot,t)-\partial_t^2u(\cdot,t),\partial_te_H(\cdot,t)\big) \quad \text{for a.e. }t\in(0,T),
\end{align}
where
\begin{align}\label{eq:energy}
E_H(t) &:= \frac12 b_H'\big(\Pi_H \partial_t e_H(\cdot,t), \Pi_H\partial_te_H(\cdot,t) \big) + \frac12 a\big( e_H(\cdot,t), e_H(\cdot,t) \big).
\end{align}
From Lemma \ref{lem:a} and Lemma \ref{lem:b}, it follows that
\begin{equation}\label{eq:energyEquiv}
\begin{aligned}
E_H^{1/2}(t) &\leq C (\|\Pi_H\partial_te_H(\cdot,t)\|_0 + \|e_H(\cdot,t)\|_1), \\
E_H^{1/2}(t) &\geq C^{-1} (\|\Pi_H\partial_te_H(\cdot,t)\|_0 + \|e_H(\cdot,t)\|_1)
\end{aligned}
\end{equation}
for any $t\in(0,T)$.
Note that $e_H|_{t=0}=\partial_te_H|_{t=0} \equiv 0$ due to the initial conditions of $u$ and $\tu_H$. If we integrate \eqref{eq:err2} over some time interval $(0,t)\subseteq (0,T)$ and apply integration by parts to the right-hand side, we obtain
\begin{align*}
E_H(t) &= r_H\big(f(\cdot,t)-\partial_t^2u(\cdot,t), e_H(\cdot,t)\big) - 
\int_0^{t} r_H\big(\partial_tf(\cdot,t')-\partial_t^3u(\cdot,t'), e_H(\cdot,t')\big) \;\dd t' \\
&\leq CH^2 \left(\|f-\partial_t^2u\|_{\infty,1} + T\|\partial_tf-\partial_t^3u\|_{\infty,1}\right) \|e_H\|_{\infty,1} \\
&\leq CH^2 \left(\|f-\partial_t^2u\|_{\infty,1} + T\|\partial_tf-\partial_t^3u\|_{\infty,1}\right) \left( \sup_{t'\in(0,T)} E_H^{1/2}(t')\right),
\end{align*}
where the second line follows from Lemma \ref{lem:rH} and the last line from \eqref{eq:energyEquiv}.
Taking the supremum over $t \in(0,T)$ and dividing by $\sup_{t\in(0,T)}E_H^{1/2}(t)$ gives
\begin{align}
\sup_{t\in(0,T)}E_H^{1/2}(t) &\leq CH^2 \left(\|f-\partial_t^2u\|_{\infty,1} + T\|\partial_tf-\partial_t^3u\|_{\infty,1}\right).
\end{align}
The lemma then follows from \eqref{eq:energyEquiv} and the triangle inequality.
\end{proof}

The following result now immediately follows from Lemma~\ref{lem:errsemidisc} and Lemma~\ref{lem:regEst}.

\begin{cor}
\label{cor:errsemidisc}
Let $u$ and $\tu_H$ be defined as in Lemma~\ref{lem:errsemidisc}. If $f$ satisfies \eqref{eq:regcond} for $k=3$ and $\partial_t^k f\in L^{\infty}(0,T;H_0^1(\Omega))$ for $k\leq 1$, then
\begin{align*}
\|\Pi_H\partial_t(u-\tu_H)\|_{\infty,0} + \|u-&\tu_H\|_{\infty,1} 
\leq CH^2\big( \|f\|_{\infty,1} + T\| \partial_tf \|_{\infty,1} + T\|\partial_t^2f\|_{\infty,0}  + T^2 \|\partial_t^3 f\|_{\infty,0} \big).
\end{align*}
\end{cor}

\subsection{Localization error}\label{ss:locerr} 

As a next step, we investigate how the localization with a parameter $\ell \in \mathbb{N}$ influences the error. Before we prove the corresponding error estimate, we require a localization result that is well known for the LOD method in an elliptic context, see, e.g., \cite{HenP13,Mai20,MalP20}.
\begin{lem}[Elliptic localization error]\label{lem:locerrEll}
Let $v_H \in U_H$ and $\ell \in \mathbb{N}$. Then
\begin{equation*}\label{eq:decay1}
|(\cS_H-\cS_{H}^\ell)v_H|_1 \leq C_\mathrm{loc}\,\ell^{d/2}\exp(-c_\mathrm{dec}\ell)|v_H|_1
\end{equation*}
with constants $c_\mathrm{dec}$ and $C_\mathrm{loc}$ that depend on $\alpha_{\max}/\alpha_{\min}$ but not on the mesh size or the scales at which $\alpha$ or $\beta$ vary. 
\end{lem}

\begin{lem}[Localization error]\label{lem:locerr}
Let $\tu_H$ be the solution to~\eqref{eq:MLLOD} for $\ell = \infty$ and $\tu_H^\ell$ the solution to~\eqref{eq:MLLOD} for $\ell \in \mathbb{N}$. Also, assume that $\partial_t^k\tu_H \in L^{\infty}(0,T;H_0^1(\Omega))$ for $k\leq 1$. Then the error $\tu_H - \tu_H^\ell$ fulfils the following error estimate,
\begin{equation*}\label{eq:locerr}
\begin{aligned}
\|\Pi_H\partial_t(\tu_H - \tu_H^\ell)\|_{\infty,0} + \|\tu_H - \tu_H^\ell\|_{\infty,1} &\leq C\,\ell^{d/2}\exp(-c_\mathrm{dec}\ell)\big(\| \tu_H\|_{\infty,1}+T\,\| \partial_t\tu_H\|_{\infty,1}\big).
\end{aligned}
\end{equation*}
\end{lem}

\begin{proof}
Let $u_H = \Pi_H \tu_H$, $u_H^\ell = \Pi_H \tu_H^\ell$, and $\be_H(\cdot,t) = u_H(\cdot,t) - u_H^\ell(\cdot,t)$. From \eqref{eq:MLLOD2}, it follows that
\begin{equation}\label{eq:locErr}
\begin{aligned}
b_H'(\partial_t^2\be_H(\cdot,t), w_H) + a(\cS_H^\ell \be_H(\cdot,t), \cS_H^\ell w_H) &= a((\cS_H^\ell-\cS_H)u_H(\cdot,t),\cS_H^\ell w_H)\\
& = a((\cS_H^\ell-\cS_H)u_H(\cdot,t),(\cS_H^\ell-\cS_H)w_H)
\end{aligned}
\end{equation}
for all $w_H \in U_H$ and $ t \in (0,T)$, where the second line follows from the orthogonality property~\eqref{eq:ortho}.
Similarly to~\eqref{eq:energy}, we define an adapted energy by
\begin{align*}\label{eq:energyAd}
\bE_H^\ell(t) &:= \frac12 b_H'\big( \partial_t \be_H(\cdot,t), \partial_t\be_H(\cdot,t) \big) + \frac12 a\big( \cS_H^\ell\be_H(\cdot,t), \cS_H^\ell \be_H(\cdot,t) \big).
\end{align*}
From Lemma~\ref{lem:a} and Lemma~\ref{lem:b}, we have that
\begin{equation}\label{eq:energyEquiv2}
\begin{aligned}
(\bE_H^{\ell})^{1/2}(t) &\leq C (\|\partial_t\be_H(\cdot,t)\|_0 + \|\cS_H^\ell \be_H(\cdot,t)\|_1), \\
(\bE_H^{\ell})^{1/2}(t) &\geq C^{-1} (\|\partial_t\be_H(\cdot,t)\|_0 + \|\cS_H^\ell \be_H(\cdot,t)\|_1),
\end{aligned}
\end{equation}
for all $t\in(0,T)$.
By substituting $w_H = \partial_t \be_H(\cdot,t)$ into~\eqref{eq:locErr}, we obtain
\begin{equation*}
\partial_t \bE^\ell_H(t) = a((\cS_H^\ell-\cS_H)u_H(\cdot,t),(\cS_H^\ell-\cS_H)\partial_t\be_H(\cdot,t)).
\end{equation*}
Integration from $0$ to $t$ and integration by parts on the right-hand side lead to
\begin{align*}
\bE_H^\ell(t) &= a((\cS_H-\cS_H^\ell)u_H(\cdot,t),(\cS_H-\cS_H^\ell)\be_H(\cdot,t))\\
&\quad - \int_0^t a((\cS_H-\cS_H^\ell)\partial_tu_H(\cdot,t'),(\cS_H-\cS_H^\ell)\be_H(\cdot,t')) \dt' \\
&\leq C \left( \| (\cS_H-\cS_H^\ell)u_H \|_{\infty,1} + T\|(\cS_H-\cS_H^\ell)\partial_tu_H\|_{\infty,1} \right) \|(\cS_H-\cS_H^\ell)\be_H\|_{\infty,1} \\
&\leq C\,\ell^{d}\exp(-c_\mathrm{dec}2\ell) \left( \| u_H \|_{\infty,1} + T\| \partial_tu_H\|_{\infty,1} \right) \|\be_H\|_{\infty,1},
\end{align*}
where the second to last line follows from Lemma~\ref{lem:a} and the Cauchy--Schwarz inequality and the last line from Lemma~\ref{lem:locerrEll} and the Poincar\'e inequality.
Note that, due to Lemma~\ref{lem:stab} and Lemma~\ref{lem:stab2}, we have
\begin{align*}
\|u_H\|_{\infty,1} &= \|\Pi_H\cS_Hu_H\|_{\infty,1} \leq C\|\cS_Hu_H\|_{\infty,1} = C\|\tu_H\|_{\infty,1} \\
\|\partial_tu_H\|_{\infty,1} &= \|\Pi_H\cS_H\partial_tu_H\|_{\infty,1} \leq C\|\cS_H\partial_tu_H\|_{\infty,1} = C\|\partial_t\tu_H\|_{\infty,1} \\
\|\be_H\|_{\infty,1} &= \|\Pi_H\cS_H^\ell\be_H\|_{\infty,1} \leq C\|\cS_H^\ell\be_H\|_{\infty,1}.
\end{align*}
We therefore get
\begin{align*}
\bE_H^\ell(t) &\leq C\,\ell^{d}\exp(-c_\mathrm{dec}2\ell) \left( \| \tu_H \|_{\infty,1} + T\| \partial_t\tu_H\|_{\infty,1} \right) \|\cS_H^\ell \be_H\|_{\infty,1} \\
&\leq C\,\ell^{d}\exp(-c_\mathrm{dec}2\ell) \left( \| \tu_H \|_{\infty,1} + T\| \partial_t\tu_H\|_{\infty,1} \right) \Big( \sup_{t'\in(0,T)} \bE_H^\ell(t')\Big)^{1/2}
\end{align*}
Taking the supremum over all $t\in(0,T)$ and dividing by $\left(\sup_{t\in(0,T)} \bE_H^\ell(t)\right)^{1/2}$ gives
\begin{align*}
\Big(\sup_{t\in(0,T)} \bE_H^\ell(t)\Big)^{1/2} &\leq C\,\ell^{d}\exp(-c_\mathrm{dec}2\ell) \left( \| \tu_H \|_{\infty,1} + T\| \partial_t\tu_H\|_{\infty,1} \right)
\end{align*}
From \eqref{eq:energyEquiv2}, it then follows that
\begin{align}\label{eq:localErr1}
\|\partial_t\be_H\|_{\infty,0} + \|\cS_H^\ell \be_H \|_{\infty,1} &\leq C\,\ell^{d}\exp(-c_\mathrm{dec}2\ell) \left( \| \tu_H \|_{\infty,1} + T\| \partial_t\tu_H\|_{\infty,1} \right).
\end{align}
Note that $\Pi_H(\partial_t \tu_H-\partial_t\tu_H^\ell) = \partial_t\be_H$ and $\tu_H - \tu_H^\ell = (\cS_H - \cS^\ell_{H}) u_{H} + \cS_H^\ell\be_H$. From the triangle inequality, \eqref{eq:localErr1}, Lemma \ref{lem:locerrEll}, and the Poincar\'e inequality, we therefore get
\begin{align*}
\|\Pi_H(\partial_t \tu_H-&\partial_t\tu_H^\ell)\|_{\infty,0} + \|\tu_H - \tu_H^\ell\|_{\infty,1} \\
&\leq \|\partial_t\be_H\|_{\infty,0} + \|\cS_H^\ell\be_H\|_{\infty,1} + \|(\cS_H - \cS^\ell_{H}) u_{H}\|_{\infty,1}\\
& \leq C\,\ell^{d}\exp(-c_\mathrm{dec}2\ell) \left( \| \tu_H \|_{\infty,1} + T\| \partial_t\tu_H\|_{\infty,1} \right) + C\,\ell^{d/2}\exp(-c_\mathrm{dec}\ell)\|\tu_H\|_{\infty,1}\\
&\leq C\,\ell^{d/2}\exp(-c_\mathrm{dec}\ell) \left( \| \tu_H \|_{\infty,1} + T\| \partial_t\tu_H\|_{\infty,1} \right).
\end{align*}
Here, we employed that in the second to last equation the factor in the second term is dominant compared to the factor in the first term. This concludes the proof.
\end{proof}

The following result now immediately follows from Lemma~\ref{lem:locerr} and Lemma~\ref{lem:regEstH}.

\begin{cor}
\label{cor:locerr}
Let $\tu_H$ and $\tU_H^\ell$ be defined as in Lemma \ref{lem:locerr}. If $f$ satisfies \eqref{eq:regcond} for $k=1$, then 
\begin{align*}
\|\Pi_H\partial_t(\tu_H - \tu_H^\ell)\|_{\infty,0} + \|\tu_H - \tu_H^\ell\|_{\infty,1} &\leq C\,\ell^{d/2}\exp(-c_\mathrm{dec}\ell) \big( T\| f\|_{\infty,0} + T^2 \| \partial_t f\|_{\infty,0} \big).
\end{align*}
\end{cor}

\subsection{Temporal error}
\label{ss:temperr}

As a final step in the error analysis, we estimate the error between the fully discrete approximation in~\eqref{eq:MLLODfull} and the semi-discrete approximation in~\eqref{eq:MLLOD2}. 
\begin{lem}[Temporal error]\label{lem:tempErr}
Let $\tu_H^{\ell}$ and $\tu_H^{\ell,n}$ denote the solutions of~\eqref{eq:MLLOD} and~\eqref{eq:MLLODfull}, respectively, and assume that the
CFL condition~\eqref{eq:CFL} holds.  
Also assume that $\partial_t^k\tu_H^\ell \in L^{\infty}(0,T;H^1_0(\Omega))$ for $k\leq 3$, $\partial_t^4\tu_H^\ell \in L^{\infty}(0,T;L^2(\Omega))$, and $\partial_t^k\tu_H^\ell |_{t=0} \equiv 0$ for $k\leq 3$. Then
\begin{equation*}\label{eq:timeerror1}
\begin{aligned}
\max_{1\leq n\leq N_t} \|\Pi_HD_{\Delta t}(\tu^\ell_H(\cdot,t^n) - \tu_H^{\ell,n})\|_0 + & \max_{0\leq n\leq N_t} \| \tu^\ell_H(\cdot,t^n) - \tu_H^{\ell,n} \|_1 \leq \\
& \qquad C \Delta t^2 \left( T \|\Pi_H\partial_t^4\tu_H\|_{\infty,0} + T \| \partial_t^3\tu_H\|_{\infty,1} \right),
\end{aligned}
\end{equation*}
where we use the notation $D_{\Delta t}v(\cdot,t^n) = \Delta^{-1}(v(\cdot,t^{n})-v(\cdot,t^{n-1}))$ and $D_{\Delta t}v^n = \Delta t^{-1}(v^{n}-v^{n-1})$.
\end{lem}
\begin{proof}
Let $e_H^n := \tu^\ell_H(\cdot,t^n) - \tu_H^{\ell,n}$ and observe that $e_H^n$ solves
\begin{equation}\label{eq:proof_temp_1}
\begin{aligned}
b_H'(\Delta t^{-2}&\Pi_H(e_H^{n+1}-2e_H^n+e_H^{n-1}),\Pi_H w_H) + a(e_H^n,w_H) \\&= 
b_H'(-\Pi_H\partial_t^2\tu_H^\ell(\cdot,t^n) + \Delta t^{-2}\Pi_H (\tu_H^\ell(\cdot,t^{n+1}) - 2\tu_H^\ell(\cdot,t^{n}) + \tu_H^\ell(\cdot,t^{n-1}) ),\Pi_H w_H)\\& =: b_H'(F^n_{\tu_H^\ell},\Pi_H w_H)
\end{aligned}
\end{equation}
for all $w_H \in \tU_H^\ell$. 
We now follow the ideas of~\cite{Chr09,Jol03}. Choosing the test function $w_H = e_H^{n+1}-e_H^{n-1}$ in~\eqref{eq:proof_temp_1}, we obtain
\begin{align}
\label{eq:eHn}
E_H^{\ell,n+1} - E_H^{\ell,n} &= \frac12 \Delta t\, b_H'(F^{n}_{\tu_H^\ell}, \Pi_H D_{\Delta t}e_H^{n+1} +  \Pi_H D_{\Delta t}e_H^{n})
\end{align}
for all $1\leq n\leq N_t-1$, where the discrete energy $E_H^n$ is given by
\begin{equation*}
E_H^{\ell,n} := \frac12 b_H'\big(\Pi_H D_{\Delta t}e_H^n, \Pi_H D_{\Delta t}e_H^n \big) + \frac12 a\big( e_H^{n-1}, e_H^n \big).
\end{equation*}
Next, we have to show that this energy is always positive.

First, note that, for any $v_H\in U_H$,
\begin{align*}
a(v_H,v_H) &= a((v_H -\cS_Hv_H)+\cS_Hv_H, (v_H -\cS_Hv_H)+\cS_Hv_H)  \\
&= a(v_H -\cS_Hv_H, v_H -\cS_Hv_H) + a(\cS_Hv_H,\cS_Hv_H) \\
&\geq a(\cS_Hv_H,\cS_Hv_H),
\end{align*}
where the second step follows from the orthogonality property~\eqref{eq:ortho}. We can then derive, for any $v_H\in U_H$,
\begin{align*}
a(\cS_{H}^\ell v_H, \cS_H^\ell v_H) &= a((\cS_{H}^\ell - \cS_H)v_H + \cS_Hv_H, (\cS_{H}^\ell - \cS_H)v_H + \cS_Hv_H) \\
&= a((\cS_{H}^\ell-\cS_H) v_H, (\cS_{H}^\ell-\cS_H) v_H) + a(\cS_H v_H, \cS_H v_H) \\
&\leq a((\cS_{H}^\ell-\cS_H) v_H, (\cS_{H}^\ell-\cS_H) v_H) + a(v_H, v_H) \\
&\leq \alpha_{\max} ( |(\cS_{H}^\ell-\cS_H) v_H|_1^2 + |v_H|_1^2) \\
&\leq \alpha_{\max} ( C_\mathrm{loc}^2 \, \ell^{d}\exp(-c_\mathrm{dec}2\ell) + 1 )|v_H|_1^2,
\end{align*}
where the second line follows from the orthogonality property~\eqref{eq:ortho} and the last line follows from the localization result in Lemma~\ref{lem:locerrEll}.
Therefore, using the inverse inequality (see, e.g., \cite{Sch98,GraHS05,Geo08})
\begin{equation}\label{eq:invIneq}
|v_H|_1 \leq C_\mathrm{inv}H^{-1}\|v_H\|_0
\end{equation}
and Lemma~\ref{lem:b}, we obtain
\begin{align*}
a(\cS_{H}^\ell v_H, \cS_H^\ell v_H) \leq C_\ell\, \alpha_{\max} C^2_\mathrm{inv}H^{-2} \|v_H\|_0^2 \leq C_\ell\, \frac{\alpha_{\max}}{\beta_{\min}} C_b^2 C^2_\mathrm{inv}H^{-2} b_H'(v_H,v_H)
\end{align*}
for all $v_H\in U_H$, where we abbreviate $C_\ell = (C_\mathrm{loc}^2\,\ell^{d}\exp(-c_\mathrm{dec}2\ell) + 1)$. This immediately implies that
\begin{align*}\label{eq:proof_temp_2}
a(\tilde v_H, \tilde v_H) \leq C_\ell\, \frac{\alpha_{\max}}{\beta_{\min}} C_b^2 C^2_\mathrm{inv}H^{-2} b_H'(\Pi_H\tilde v_H,\Pi_H\tilde v_H) \qquad\text{for all }\tilde v_H\in\tU_H^\ell.
\end{align*}
Using this estimate and the CFL condition~\eqref{eq:CFL}, we compute
\begin{equation*}\label{eq:proof_temp_3}
\begin{aligned}
E_H^{\ell,n} &= \frac12 \|\Pi_H D_{\Delta t}e_H^n\|_{b_H'}^{2} - \frac{1}{4}\,a( e_H^n - e_H^{n-1}, e_H^n - e_H^{n-1}) + \frac{1}{4}\,a( e_H^n, e_H^n) + \frac{1}{4}\,a( e_H^{n-1}, e_H^{n-1}) \\
&= \frac12 \|\Pi_H D_{\Delta t}e_H^n\|_{b_H'}^{2} - \frac{1}{4}\Delta t^2 \,a( D_{\Delta t}e_H^n , D_{\Delta t}e_H^n) + \frac{1}{4}\,a( e_H^n, e_H^n) + \frac{1}{4}\,a( e_H^{n-1}, e_H^{n-1}) \\
& \geq \frac12 \left( 1 - \frac12 C_\ell\,\frac{\alpha_{\max}}{\beta_{\min}} C_b^2 C^2_\mathrm{inv}H^{-2}\Delta t^2 \right) \|\Pi_H D_{\Delta t}e_H^n\|_{b_H'}^2  + \frac{1}{4} \|e_H^n\|_a^2 + \frac{1}{4}\|e_H^{n-1}\|_a^2 \\
& \geq \frac{\delta}{2}\, \|\Pi_H D_{\Delta t}e_H^n\|_{b_H'}^2 + \frac{1}{4} \|e_H^n\|_a^2 + \frac{1}{4}\|e_H^{n-1}\|_a^2 \geq 0
\end{aligned}
\end{equation*}
for all $1\leq n \leq N_t$. From this, Lemma \ref{lem:a}, and Lemma \ref{lem:b}, it follows that
\begin{subequations}
\label{eq:equivEhn}
\begin{align}
\big( E_h^{\ell,n}\big)^{1/2} &\leq C(\| \Pi_H D_{\Delta t} e_H^n\|_0 + \| e_H^n \|_1 + \| e_H^{n-1} \|_1), \\
\big( E_h^{\ell,n}\big)^{1/2} &\geq C^{-1}(\| \Pi_H D_{\Delta t} e_H^n\|_0 + \| e_H^n \|_1 + \|e_H^{n-1}\|_1),
\end{align}
\end{subequations}
for all $1\leq n \leq N_t$.
Using \eqref{eq:eHn}, the Cauchy--Schwarz inequality, the triangle inequality, Lemma~\ref{lem:b}, and~\eqref{eq:equivEhn}, we now obtain
\begin{align*}
E_H^{\ell,n+1} - E_H^{\ell,n} &\leq \frac12 \Delta t\, \|F^n_{\tu_H^\ell}\|_{b_H'} \,\big(\|\Pi_H D_{\Delta t}e_H^{n+1}\|_{b_H'} +  \|\Pi_H D_{\Delta t}e_H^{n}\|_{b_H'}\big)\\
&\leq C \Delta t\, \|F^n_{\tu_H^\ell}\|_{0} \big(\big(E_H^{\ell,n+1}\big)^{1/2} + \big(E_H^{\ell,n}\big)^{1/2}\big) \\
&\leq C \Delta t \left( \max_{1\leq n'\leq N_t-1} \|F^{n'}_{\tu_H^\ell}\|_{0} \right) \left( \max_{1\leq n'\leq N_t} \big( E_H^{\ell,n'}\big)^{1/2} \right)
\end{align*}
for all $1\leq n \leq N_t-1$.
Using this recursive relation and \eqref{eq:equivEhn}, we can obtain
\begin{align*}
E_H^{\ell,n+1} &\leq E_h^{\ell,1} + C n \Delta t \left( \max_{1\leq n'\leq N_t-1} \|F^{n'}_{\tu_H^\ell}\|_{0} \right) \left( \max_{1\leq n'\leq N_t} \big( E_H^{\ell,n'}\big)^{1/2} \right) \\
&\leq \left( \big(E_h^{\ell,1}\big)^{1/2} + C T \left( \max_{1\leq n'\leq N_t-1} \|F^{n'}_{\tu_H^\ell}\|_{0} \right) \right) \left( \max_{1\leq n'\leq N_t} \big( E_H^{\ell,n'}\big)^{1/2} \right)\\
&\leq C \left( \|\Pi_H D_{\Delta t} e_H^1\|_0 + \|e_H^1\|_1 + C T \left( \max_{1\leq n' \leq N_t-1} \|F^{n'}_{\tu_H^\ell}\|_0 \right) \right) \left( \max_{1\leq n'\leq N_t} \big( E_H^{\ell,n'}\big)^{1/2} \right)
\end{align*} 
for all $0\leq n\leq N_t-1$. Taking the maximum over all $0\leq n\leq N_t-1$ and dividing by $\max_{1\leq n\leq N_t} \big( E_H^{\ell,n}\big)^{1/2}$ then gives 
\begin{equation*}
\max_{1\leq n\leq N_t} \big( E_H^{\ell,n}\big)^{1/2}  \leq C \left( \|\Pi_H D_{\Delta t} e_H^1\|_0 + \|e_H^1\|_1 + T \left( \max_{1\leq n \leq N_t-1} \|F^{n}_{\tu_H^\ell}\|_0 \right) \right).
\end{equation*}
From \eqref{eq:equivEhn}, it then follows that 
\begin{equation}\label{eq:timeerror2}
\begin{aligned}
\max_{1\leq n\leq N_t} \| \Pi_H D_{\Delta t} e_H^n\|_0 + & \max_{0\leq n\leq N_t} \| e_H^n \|_1  \\&\leq 
C \left( \| \Pi_H D_{\Delta t} e_H^1\|_0 + \|e_H^1\|_1 + T \left( \max_{1\leq n \leq N_t-1} \|F^{n}_{\tu_H^\ell}\|_0 \right) \right).
\end{aligned}
\end{equation}
It remains to estimate $\|e_H^1\|_1$ and $\|F^n_{\tu_H^\ell}\|_{0}$. Using the assumption that $\partial_t^k\tu_H^{\ell}|_{t=0}\equiv 0$ for $k\leq 3$ as well as a third-order and a second-order Taylor expansion around $t=0$, we obtain 
\begin{align*}
\| \Pi_H D_{\Delta t}e_H^1\|_0 &= \frac{1}{\Delta t}\| \Pi_H \tu_H^{\ell}(\cdot,t^1) \|_0 \leq C\Delta t^3 \|\Pi_H\partial_t^4 \tu_H^{\ell}\|_{\infty,0}, \\
\| e_H^1\|_1 &= \|\tu_H^\ell(\cdot,t^1)\|_1 \leq C\Delta t^3 \| \partial_t^3\tu_H^{\ell}\|_{\infty,1}.
\end{align*}
Using a third-order Taylor expansion around $t=t^n$, we also get, for all $1\leq n \leq N_t-1$, 
\begin{align*}
\|F^{n}_{\tu_H^\ell}\|_0 &= \|\Pi_H\partial_t^2\tu_H^\ell(\cdot,t^n)- \Delta t^{-2}\Pi_H(\tu_H^\ell(\cdot,t^{n+1}) - 2\tu_H^\ell(\cdot,t^{n}) + \tu_H^\ell(\cdot,t^{n-1}))\|_0 \\
&\leq C\,\Delta t^2\| \Pi_H \partial_t^4 \tu_H^\ell\|_{\infty,0}.
\end{align*}
Together with \eqref{eq:timeerror2}, these last three estimates prove the assertion.
\end{proof}

The following result now directly follows from Lemma~\ref{lem:tempErr} and Lemma~\ref{lem:regEstH}.

\begin{cor}
\label{cor:tempErr}
Let $\tu_H^\ell$ and $\tu_H^{\ell,n}$ be defined as in Lemma \ref{lem:tempErr}. If $f$ satisfies \eqref{eq:regcond} for $k=3$, then
\begin{align*}
\max_{1\leq n\leq N_t} \|\Pi_HD_{\Delta t}(\tu^\ell_H(\cdot,t^n) - \tu_H^{\ell,n})\|_0 + \max_{0\leq n\leq N_t} \| \tu^\ell_H(\cdot,t^n) - \tu_H^{\ell,n} \|_1 &\leq C \Delta t^2 T^2 \|\partial_t^3f\|_{\infty,0}.
\end{align*}
\end{cor}

\begin{figure}
	\centering
	\scalebox{0.92}{
\begin{tikzpicture}

\begin{axis}[
colorbar,
colorbar style={ylabel={}},
colormap/viridis,
hide x axis,
hide y axis,
point meta max=4,
point meta min=0.5,
tick align=outside,
tick pos=left,
x grid style={white!69.0196078431373!black},
xmin=-0.5, xmax=127.5,
xtick style={color=black},
y grid style={white!69.0196078431373!black},
ymin=-0.5, ymax=127.5,
ytick style={color=black}
]
\addplotgraphicsnatural
[includegraphics cmd=\pgfimage,xmin=-0.5, xmax=127.5, ymin=-0.5, ymax=127.5] {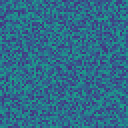};
\end{axis}

\end{tikzpicture}
	}
	\hspace{2em}
	\scalebox{0.92}{
\begin{tikzpicture}

\begin{axis}[
colorbar,
colorbar style={ylabel={}},
colormap/viridis,
hide x axis,
hide y axis,
point meta max=4,
point meta min=0.5,
tick align=outside,
tick pos=left,
x grid style={white!69.0196078431373!black},
xmin=-0.5, xmax=127.5,
xtick style={color=black},
y grid style={white!69.0196078431373!black},
ymin=-0.5, ymax=127.5,
ytick style={color=black}
]
\addplotgraphicsnatural
[includegraphics cmd=\pgfimage,xmin=-0.5, xmax=127.5, ymin=-0.5, ymax=127.5] {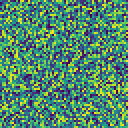};
\end{axis}

\end{tikzpicture}
	}
	\caption
	{\small Coefficients $\alpha$ (left) and $\beta$ (right) for Example~1.} 
	\label{fig:coeff1}
\end{figure}

\subsection{Proof of Theorem~\ref{thm:fullerr}}
\label{ss:proofthm}

\begin{proof}[Proof of Theorem~\ref{thm:fullerr}]
Let $\tu_H^\ell$ be the solution of~\eqref{eq:LOD} and $\tu_H = \tu_H^\infty$ the non-localized solution of~\eqref{eq:LOD}. 
Define the errors $e_H:= u-\tu_H$, $\be_H:=\tu_H-\tu_H^\ell$, and $e_H^n:=\tu_H^\ell(\cdot,t^n) - \tu_H^{\ell,n}$. Note that
\begin{align*}
u(\cdot,t^n)-\tu_H^{\ell,n} &= e_H(\cdot,t^n) + \be_H(\cdot,t^n) + e_H^n, \\
D_{\Delta t}u - D_{\Delta t}\tu_H^{\ell,n} &= D_{\Delta t}e_H(\cdot,t^n) + D_{\Delta t}\be_H(\cdot,t^n) + D_{\Delta t}e_H^n
\end{align*}
for any $1 \leq n \leq N_t$. 
Furthermore, since
\begin{align*}
D_{\Delta t}e_H(\cdot,t^n) &= \Delta t^{-1} \int_{t^{n-1}}^{t^n} \partial_t e_H(\cdot,\tau) \,\dd\tau, \qquad
D_{\Delta t}\be_H(\cdot,t^n) = \Delta t^{-1} \int_{t^{n-1}}^{t^n} \partial_t\be_H(\cdot,\tau) \,\dd\tau,
\end{align*}
we have
\begin{align*}
\| \Pi_H D_{\Delta t}e_H(\cdot,t^n) \|_0 &\leq \| \Pi_H \partial_t e_H \|_{\infty,0}, \qquad
\| \Pi_H D_{\Delta t}\be_H(\cdot,t^n) \|_0 \leq \| \Pi_H \partial_t \be_H \|_{\infty,0}.
\end{align*}
Finally, from Corollary \ref{cor:errsemidisc}, Corollary \ref{cor:locerr}, and Corollary \ref{cor:tempErr}, it follows that
\begin{align*}
\|\Pi_H\partial_te_H\|_{\infty,0} + \| e_H\|_{\infty,1} &\leq CH^2(\| f \|_{\infty,1} + T \| \partial_t f \|_{\infty,1} + T \| \partial_t^2 f \|_{\infty,0} + T^2 \| \partial_t^3 f \|_{\infty,0}), \\
\|\Pi_H\partial_t\be_H\|_{\infty,0} + \|\be_H\|_{\infty,1} &\leq  C\, \ell^{d/2}\exp(-c_\mathrm{dec}\ell) (T\| f \|_{\infty,0} + T^2 \| \partial_t f \|_{\infty,0}), \\
\max_{1\leq n\leq N_t}\|\Pi_H D_{\Delta t}e_H^n \|_0 + \max_{0\leq n\leq N_t} \|e_H^n\|_1 &\leq C\Delta t^2 T^2 \| \partial_t^3 f \|_{\infty,0}.
\end{align*}
The theorem now follows from the triangle inequality and the estimates above.
\end{proof}

\section{Numerical examples}\label{s:numerics}

In this section, we present numerical experiments to confirm the theoretical results shown in the previous subsections. The computations are performed with Python using an adapted version of the software \texttt{gridlod}~\cite{HelK19}. 

Before we turn to the experiments, we first discuss a last important aspect in order to obtain a fully practical method in the next subsection.

\subsection{Discretisation at the fine scale}

It is important to note that we have assumed so far that the corrections $\cQ_e^{\ell}v_H$, for $e \in \cT_H$ and $v_H \in U_H$, as well the global correction $\cQ^\ell v_H$ can be computed exactly. In practice, however, the corresponding problems given in~\eqref{eq:loccor} need to be discretised as well. A very straight-forward approach in this regard is to replace the space $H^1_0(\Omega)$ in the above construction by a classical first-order finite element space $U_h$ with some fine mesh size $h$ that resolves oscillations of the coefficients. We emphasize that the above error analysis follows verbatim. More precisely, let $u_h$ be the solution of~\eqref{eq:FEM} on the scale $h$ and let $u_{H,h}^{\ell,n},\,n\geq 0$ be the solution of~\eqref{eq:MLLODfull} where the corresponding correction operator $\cS_{H,h}^\ell$ is defined in the kernel space $W_{H,h} = \ker \Pi_H\vert_{U_h}$. Then the result in Theorem~\ref{thm:fullerr} instead provides an error estimate of the form
\begin{equation*}
\|D_{\Delta t}u_h(\cdot,t^n) - D_{\Delta t}u_{H,h}^{\ell,n}\|_0 + \|u_h(\cdot,t^n) - \cS_{H,h}^{\ell}u_{H,h}^{\ell,n}\|_1
\leq C\,(H^2 + \Delta t^2 + \ell^{d/2}\exp(-c_\mathrm{dec}\ell)) \cdot \texttt{data}.
\end{equation*}
This is the error that we measure for the experiments in the following subsections. That is, we investigate the error between the LOD solution in~\eqref{eq:MLLODfull} and a fine enough finite element approximation. 

\subsection{Example 1: random coefficients}

\begin{figure}
	\centering
	\scalebox{.85}{
\begin{tikzpicture}

\begin{axis}[
legend cell align={left},
legend style={fill opacity=0.8, draw opacity=1, text opacity=1, at={(0.03,0.97)}, anchor=north west, draw=white!80!black},
legend pos = south east,
log basis x={10},
log basis y={10},
tick align=outside,
tick pos=left,
x grid style={white!69.0196078431373!black},
xlabel={mesh size \(\displaystyle H\)},
xmin=0.0131390064883393, xmax=0.594603557501361,
xmode=log,
xtick style={color=black},
y grid style={white!69.0196078431373!black},
ylabel={rel error in $\|\cdot\|_{\infty,1}$},
ymin=0.000163412946118649, ymax=1.11960669956248,
ymode=log,
ytick style={color=black}
]
\addplot [semithick, myBlue, mark=square, mark size=1, mark options={solid}]
table {%
0.5 0.471127149256371
0.25 0.137076430876148
0.125 0.035182123308742
0.0625 0.0091818255320835
0.03125 0.00400344971604165
0.015625 0.00403187089477385
};
\addlegendentry{\scriptsize LOD$_\beta$ $\ell = 2$}
\addplot [semithick, myRed, mark=triangle, mark size=1, mark options={solid}]
table {%
0.5 0.753895215861838
0.25 0.258272270679268
0.125 0.0685412395084464
0.0625 0.0171970554428598
0.03125 0.00456840823588087
0.015625 0.00407042635275452
};
\addlegendentry{\scriptsize MLLOD$_\beta$ $\ell = 2$}
\addplot [semithick, myBlue, mark=square, mark size=2.5, mark options={solid}]
table {%
0.5 0.471127149256371
0.25 0.135855516467263
0.125 0.0345315387064656
0.0625 0.00856182678721979
0.03125 0.00189561665721857
0.015625 0.00166863328027615
};
\addlegendentry{\scriptsize LOD$_\beta$ $\ell = 3$}
\addplot [semithick, myRed, mark=triangle, mark size=3, mark options={solid}]
table {%
0.5 0.753895215861838
0.25 0.258962327832991
0.125 0.069262986807045
0.0625 0.0174326823973546
0.03125 0.0044391340170682
0.015625 0.00174500954917601
};
\addlegendentry{\scriptsize MLLOD$_\beta$ $\ell = 3$}
\addplot [semithick, myBlue, mark=square, mark size=4, mark options={solid}]
table {%
0.5 0.471127149256371
0.25 0.135855516467263
0.125 0.0343928535767906
0.0625 0.00858335390635077
0.03125 0.00195679660081537
0.015625 0.000343744057087114
};
\addlegendentry{\scriptsize LOD$_\beta$ $\ell = 4$}
\addplot [semithick, myRed, mark=triangle, mark size=5, mark options={solid}]
table {%
0.5 0.753895215861838
0.25 0.258962327832991
0.125 0.0693420826409513
0.0625 0.0174894948596606
0.03125 0.00448530080021982
0.015625 0.00122096084832821
};
\addlegendentry{\scriptsize MLLOD$_\beta$ $\ell = 4$}
\addplot [semithick, black, dashed]
table {%
	0.5 0.25
	0.25 0.0625
	0.125 0.015625
	0.0625 0.00390625
	0.03125 0.0009765625
	0.015625 0.000244140625
};
\addlegendentry{\scriptsize order 2}
\end{axis}

\end{tikzpicture}
	}
	\scalebox{.85}{
\begin{tikzpicture}

\begin{axis}[
legend cell align={left},
legend style={fill opacity=0.8, draw opacity=1, text opacity=1, at={(0.03,0.97)}, anchor=north west, draw=white!80!black},
legend pos = south east,
log basis x={10},
log basis y={10},
tick align=outside,
tick pos=left,
x grid style={white!69.0196078431373!black},
xlabel={mesh size \(\displaystyle H\)},
xmin=0.0131390064883393, xmax=0.594603557501361,
xmode=log,
xtick style={color=black},
y grid style={white!69.0196078431373!black},
ylabel={rel error in $\|\cdot\|_{\infty,1}$},
ymin=0.000163412946118649, ymax=1.11960669956248,
ymode=log,
ytick style={color=black}
]
\addplot [thick, myBlue, mark=square, mark size=4, mark options={solid}]
table {%
0.5 0.473306973969103
0.25 0.135645142066204
0.125 0.0379196557486266
0.0625 0.0158923086255158
0.03125 0.00937487252274897
0.015625 0.00551241296893147
};
\addlegendentry{\scriptsize LOD $\ell = 4$}
\addplot [thick, myRed, mark=triangle, mark size=4, mark options={solid}]
table {%
0.5 0.754429306289769
0.25 0.25744931651557
0.125 0.0692414107688945
0.0625 0.0186920876335358
0.03125 0.00939808116908686
0.015625 0.00555346056613528
};
\addlegendentry{\scriptsize MLLOD $\ell = 4$}
\addplot [thick, myGreen, mark=o, mark size=4, mark options={solid}]
table {%
0.5 0.561654326815114
0.25 0.293722093015143
0.125 0.204123521816957
0.0625 0.175963666188048
0.03125 0.155558264211382
0.015625 0.0730269721256925
};
\addlegendentry{\scriptsize FEM}
\addplot [semithick, black, dashed]
table {%
	0.5 0.5
	0.25 0.25
	0.125 0.125
	0.0625 0.0625
	0.03125 0.03125
	0.015625 0.015625
};
\addplot [semithick, black, dashed]
table {%
	0.5 0.25
	0.25 0.0625
	0.125 0.015625
	0.0625 0.00390625
	0.03125 0.0009765625
	0.015625 0.000244140625
};
\addlegendentry{\scriptsize order 1 and 2}
\end{axis}

\end{tikzpicture}
	}
	\caption{\small Relative errors for Example~1 with respect to $H$ for different spatial discretisation approaches. Left: Correctly weighted LOD approach and its lumped version. Right: Classical finite element approximation and naively weighted LOD approach (both lumped and non-lumped).}
	\label{fig:ex1}
\end{figure}
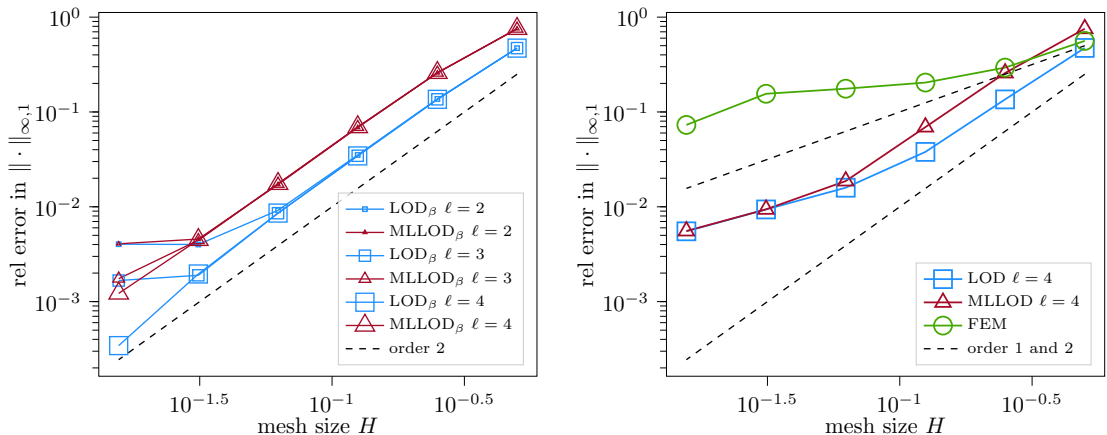

For the first experiment, we consider coefficients $\alpha$ and $\beta$ that are obtained as independently and uniformly distributed random values on a finite element mesh $\cT_\varepsilon$ with mesh size $\varepsilon = 2^{-6}$. The values are chosen in $[1,2.5]$ and $[0.5,4]$ for $\alpha$ and $\beta$, respectively, and the actual samples are depicted in Figure~\ref{fig:coeff1}. We choose the right-hand side
\begin{equation*}
f(x,t) = \sin(\pi\,x_1)\,\sin(\pi\,x_2)\,\cos(0.5\pi\,t),
\end{equation*}
the final time $T = 1$, and compute a reference solution with a classical leapfrog scheme combined with a first-order finite element method on a reference mesh $\cT_h$ with mesh size $h = 2^{-7}$ that resolves the fine oscillations in the coefficients. Although the function $f$ does not fulfil the condition~\eqref{eq:regcond_f}, this is not critical for the convergence behavior as emphasized in Remark~\ref{rem:initial}.
The step size is chosen as $\Delta t = 0.25\,h$, which is a rather sharp condition for the stability of the method. The same relation is also employed for the mass lumped LOD method. The relative errors in the $L^\infty(0,T;H^1(\Omega))$ norm for this approach are presented in Figure~\ref{fig:ex1} (left, {\protect \tikz{ \draw[line width=1.5pt, myRed] (0,0) -- (0.13,0.23) -- (0.26, 0) -- cycle;}} ) for different values of $\ell$. The curves show the expected second-order convergence with a commencing stagnation if $\ell$ is not increased appropriately. For comparison, we also present the errors for the method without mass lumping in Figure~\ref{fig:ex1} (left, {\protect \tikz{ \draw[line width=1.5pt, myBlue] (0,0) rectangle (0.23,0.23);}}). The method behaves similarly but overall leads to slightly smaller errors. 
To investigate the speed-up that can be achieved with the lumping strategy, we present computation times (in seconds) for different values of $H$ in the following table.\\[-2ex]
\newcolumntype{P}{R{1.4cm}}
\begin{center}
\begin{tabular}{rPPP}
	$H$ & $2^{-4}$ & $2^{-5}$ & $2^{-6}$ 
	\\\toprule
	offline assembly & 111.510 & 441.733 & 1819.002
	\\\midrule
	lumped LOD$_\beta$ (online) & 0.008 & 0.038 & 0.417
	\\
	non-lumped LOD$_\beta$ (online) & 0.267 & 4.561 & 73.120
	\\
	\textbf{speed-up factor} & 33.375 & 120.026 & 175.348
\end{tabular}\\[1.5ex]
\end{center}
Both the non-lumped and the lumped version need to compute LOD matrices in an \emph{offline phase} and therefore require similar computation times. The main advantage of the lumped version, however, is in the \emph{online phase}, because of its fully explicit nature. There, the lumped method achieves a significant speed-up compared to the non-lumped method and is therefore especially beneficial when simulating on longer time intervals or for multiple different right-hand sides for instance.

To illustrate the advantage of our approach with an appropriate weighting of $\beta$ in the definition of the interpolation operator $\Pi_H$, we further present the errors for a naive LOD approach without a coefficient-dependent averaging. In Figure~\ref{fig:ex1} (right) we present the errors obtained with ({\protect \tikz{ \draw[line width=1.5pt, myRed] (0,0) -- (0.13,0.23) -- (0.26, 0) -- cycle;}}) and without ({\protect \tikz{ \draw[line width=1.5pt, myBlue] (0,0) rectangle (0.23,0.23);}}) mass lumping as well as with a standard finite element approach ({\protect \tikz{ \draw[line width=1.5pt, myGreen] circle (0.6ex);}}). One observes that the finite element method shows a stagnation of the error in the \emph{pre-asymptotic} regime where the oscillations of the coefficients are not yet resolved. The results of the two multiscale approaches first indicate a second-order rate for coarse mesh sizes but then show a reduced convergence rate of only order 1 due to the sub-optimal averaging strategy. 

\subsection{Example 2: structured coefficients} 

\begin{figure}
	\centering
	\scalebox{0.92}{
\begin{tikzpicture}

\begin{axis}[
colorbar,
colorbar style={ylabel={}},
colormap/viridis,
hide x axis,
hide y axis,
point meta max=18,
point meta min=1,
tick align=outside,
tick pos=left,
x grid style={white!69.0196078431373!black},
xmin=-0.5, xmax=127.5,
xtick style={color=black},
y grid style={white!69.0196078431373!black},
ymin=-0.5, ymax=127.5,
ytick style={color=black}
]
\addplotgraphicsnatural
[includegraphics cmd=\pgfimage,xmin=-0.5, xmax=127.5, ymin=-0.5, ymax=127.5] {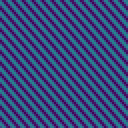};
\end{axis}

\end{tikzpicture}
	}
	\hspace{2em}
	\scalebox{0.92}{
\begin{tikzpicture}

\begin{axis}[
colorbar,
colorbar style={ylabel={}},
colormap/viridis,
hide x axis,
hide y axis,
point meta max=18,
point meta min=1,
tick align=outside,
tick pos=left,
x grid style={white!69.0196078431373!black},
xmin=-0.5, xmax=127.5,
xtick style={color=black},
y grid style={white!69.0196078431373!black},
ymin=-0.5, ymax=127.5,
ytick style={color=black}
]
\addplotgraphicsnatural
[includegraphics cmd=\pgfimage,xmin=-0.5, xmax=127.5, ymin=-0.5, ymax=127.5] {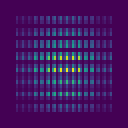};
\end{axis}

\end{tikzpicture}
	}
	\caption
	{\small Coefficients $\alpha$ (left) and $\beta$ (right) for Example~2.} 
	\label{fig:coeff2}
\end{figure}

In a second example, we investigate the behavior of the method for more structured coefficients $\alpha$ and $\beta$ as depicted in Figure~\ref{fig:coeff2}. Again, these coefficients are piece-wise constant on a mesh $\cT_\varepsilon$ with $\varepsilon = 2^{-6}$. We set $T=1$,
\begin{equation*}
f(x,t) = \sin(3\pi\,x_1)\,x_2(1-x_2)\,t^2,
\end{equation*}
and we compute a reference solution on the scale $h = 2^{-7}$. The CFL condition is slightly stronger for this example due to the larger bounds of the coefficient and the maximal choice for the time step size therefore reads $\Delta t = 0.15\,H$ for some given mesh size $H$.
The results are very similar to the ones in Example~1. With a correct weighting in the definition of the interpolation operator $\Pi_H$ and correctly adjusted localization parameter $\ell$, both the lumped and the non-lumped version of the LOD approach again show a second-order convergence behaviour (cf.~Figure~\ref{fig:ex2} (left)), while the $\beta$-independent definition of $\Pi_H$ leads to a reduced convergence rate, already for coarse mesh sizes (cf.~Figure~\ref{fig:ex2} (right)). Moreover, the finite element approach again shows basically no improvement when decreasing $H$ and only starts to show a linear behaviour when the scale on which the coefficients vary is resolved.

\subsection{Example 3: larger contrast}	

Our theoretical results include a dependence of the constants~$C$ and~$c_\mathrm{dec}$ in~\eqref{eq:finalerr} on the contrasts $\max\{\alpha_{\max},\beta_{\max}\}/\min\{\alpha_{\min},\beta_{\min}\}$ and $\alpha_{\max}/\alpha_{\min}$, respectively. To investigate how the mass lumped LOD approach performs for larger contrast, we revisit Example~1 and rescale the coefficients $\alpha$ and $\beta$ (cf.~Figure~\ref{fig:coeff1}) both to the range $[0.01,100]$.  
In this case, the CFL condition reads $\Delta t = 0.01\, H$. The corresponding errors are presented in Figure~\ref{fig:ex3}. 
The results indicate the same convergence rates as for Example~1 in Figure~\ref{fig:ex1} and the errors are roughly of the same size, indicating that the dependence of the error on the contrasts as mentioned in Theorem~\ref{thm:fullerr} is most likely not severe. Concerning the exponentially decaying localization term, one can observe that the threshold where the localization error dominates appears to be larger than in Figure~\ref{fig:ex1}. However, an increase of $\ell$ compared to Example~1 seems not to be required to preserve the second-order convergence rates in this case. 

\section{Conclusion}\label{s:conclusion}

In this work, we have introduced and analysed a special coefficient-dependent mass lumping strategy that, combined with the multiscale technique Localized Orthogonal Decomposition, leads to a method with optimal second-order convergence rates in space and time while being fully explicit. In the construction of both the multiscale approach and the diagonalization strategy of the mass matrix, the key feature is the use of an appropriate interpolation operator. We have provided a full error analysis of the mass lumped multiscale method and investigated its performance in numerical experiments. 

\begin{figure}
	\centering
	\scalebox{.85}{
\begin{tikzpicture}

\begin{axis}[
legend cell align={left},
legend style={fill opacity=0.8, draw opacity=1, text opacity=1, at={(0.03,0.97)}, anchor=north west, draw=white!80!black},
legend pos = south east,
log basis x={10},
log basis y={10},
tick align=outside,
tick pos=left,
x grid style={white!69.0196078431373!black},
xlabel={mesh size \(\displaystyle H\)},
xmin=0.0131390064883393, xmax=0.594603557501361,
xmode=log,
xtick style={color=black},
y grid style={white!69.0196078431373!black},
ylabel={rel error in $\|\cdot\|_{\infty, 1}$},
ymin=0.000155799411140302, ymax=3.04929256781561,
ymode=log,
ytick style={color=black}
]
\addplot [semithick, myBlue, mark=square, mark size=1, mark options={solid}]
table {%
0.5 1.03310232813742
0.25 0.397061380273217
0.125 0.149281073480482
0.0625 0.0337882926020129
0.03125 0.013668553350683
0.015625 0.00924003882989052
};
\addlegendentry{\scriptsize LOD$_\beta$ $\ell = 2$}
\addplot [semithick, myRed, mark=triangle, mark size=1, mark options={solid}]
table {%
0.5 1.51577524729986
0.25 0.785559626219533
0.125 0.143595284598156
0.0625 0.0435535328393217
0.03125 0.0153583238803019
0.015625 0.0101412892997338
};
\addlegendentry{\scriptsize MLLOD$_\beta$ $\ell = 2$}
\addplot [semithick, myBlue, mark=square, mark size=2.5, mark options={solid}]
table {%
0.5 1.03310232813742
0.25 0.39701665582728
0.125 0.147992428065542
0.0625 0.0324084964256991
0.03125 0.00787215831725271
0.015625 0.00319911163833066
};
\addlegendentry{\scriptsize LOD$_\beta$ $\ell = 3$}
\addplot [semithick, myRed, mark=triangle, mark size=3, mark options={solid}]
table {%
0.5 1.51577524729986
0.25 0.786474259694586
0.125 0.142172931387324
0.0625 0.0437387452099015
0.03125 0.0113985366219013
0.015625 0.00555117707552303
};
\addlegendentry{\scriptsize MLLOD$_\beta$ $\ell = 3$}
\addplot [semithick, myBlue, mark=square, mark size=4, mark options={solid}]
table {%
0.5 1.03310232813742
0.25 0.39701665582728
0.125 0.148031372983917
0.0625 0.0322046584474502
0.03125 0.00775420583989854
0.015625 0.00144827901933979
};
\addlegendentry{\scriptsize LOD$_\beta$ $\ell = 4$}
\addplot [semithick, myRed, mark=triangle, mark size=5, mark options={solid}]
table {%
0.5 1.51577524729986
0.25 0.786474259694586
0.125 0.142075367377604
0.0625 0.0436838733827158
0.03125 0.0114376404571089
0.015625 0.00468353457868396
};
\addlegendentry{\scriptsize MLLOD$_\beta$ $\ell = 4$}
\addplot [semithick, black, dashed]
table {%
0.5 0.25
0.25 0.0625
0.125 0.015625
0.0625 0.00390625
0.03125 0.0009765625
0.015625 0.000244140625
};
\addlegendentry{\scriptsize order 2}
\end{axis}

\end{tikzpicture}
	}
	\scalebox{.85}{
\begin{tikzpicture}

\begin{axis}[
legend cell align={left},
legend style={fill opacity=0.8, draw opacity=1, text opacity=1, at={(0.03,0.97)}, anchor=north west, draw=white!80!black},
legend pos = south east,
log basis x={10},
log basis y={10},
tick align=outside,
tick pos=left,
x grid style={white!69.0196078431373!black},
xlabel={mesh size \(\displaystyle H\)},
xmin=0.0131390064883393, xmax=0.594603557501361,
xmode=log,
xtick style={color=black},
y grid style={white!69.0196078431373!black},
ylabel={rel error in $\|\cdot\|_{\infty, 1}$},
ymin=0.000155799411140302, ymax=3.04929256781561,
ymode=log,
ytick style={color=black}
]
\addplot [thick, myBlue, mark=square, mark size=4, mark options={solid}]
table {%
0.5 1.10070955624898
0.25 0.377269202917797
0.125 0.178973597277727
0.0625 0.068338229205322
0.03125 0.0450177913905213
0.015625 0.0188621808860656
};
\addlegendentry{\scriptsize LOD $\ell = 4$}
\addplot [thick, myRed, mark=triangle, mark size=4, mark options={solid}]
table {%
0.5 1.5997494511748
0.25 0.727173914855425
0.125 0.169391458855053
0.0625 0.0731538788228532
0.03125 0.0457358917678072
0.015625 0.0193740556078491
};
\addlegendentry{\scriptsize MLLOD $\ell = 4$}
\addplot [thick, myGreen, mark=o, mark size=4, mark options={solid}]
table {%
0.5 1.10228599214564
0.25 0.607252181424766
0.125 0.519776462743556
0.0625 0.461938800694026
0.03125 0.381209889198713
0.015625 0.165183054848447
};
\addlegendentry{\scriptsize FEM}
\addplot [semithick, black, dashed]
table {%
	0.5 0.5
	0.25 0.25
	0.125 0.125
	0.0625 0.0625
	0.03125 0.03125
	0.015625 0.015625
};
\addplot [semithick, black, dashed]
table {%
0.5 0.25
0.25 0.0625
0.125 0.015625
0.0625 0.00390625
0.03125 0.0009765625
0.015625 0.000244140625
};
\addlegendentry{\scriptsize order 1 and 2}
\end{axis}

\end{tikzpicture}
	}
	\caption{\small Relative errors for Example~2 with respect to $H$ for different spatial discretisation approaches. Left: Correctly weighted LOD approach and its lumped version. Right: Classical finite element approximation and naively weighted LOD approach (both lumped and non-lumped).}
	\label{fig:ex2}
\end{figure}
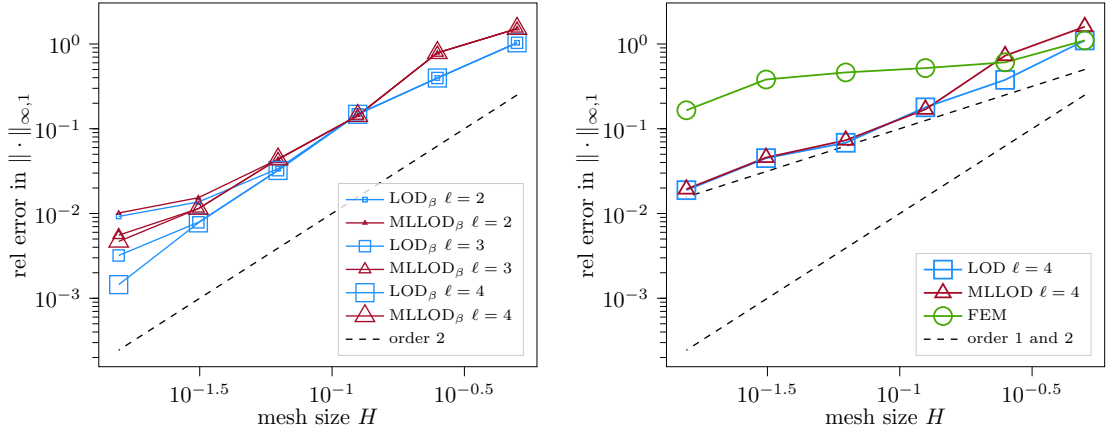

\subsection*{Acknowledgement} 
Sjoerd Geevers has been supported by the Austrian Science Fund (FWF) through the project M 2949-N.
Roland Maier gratefully acknowledges support by the G\"oran Gustafsson Foundation for Research in Natural Sciences and Medicine.

\begin{figure}
	\centering
	\scalebox{.85}{
\begin{tikzpicture}

\begin{axis}[
legend cell align={left},
legend style={fill opacity=0.8, draw opacity=1, text opacity=1, at={(0.03,0.97)}, anchor=north west, draw=white!80!black},
legend pos = south east,
log basis x={10},
log basis y={10},
tick align=outside,
tick pos=left,
x grid style={white!69.0196078431373!black},
xlabel={mesh size \(\displaystyle H\)},
xmin=0.0131390064883393, xmax=0.594603557501361,
xmode=log,
xtick style={color=black},
y grid style={white!69.0196078431373!black},
ylabel={rel error in $\|\cdot\|_{\infty,1}$},
ymin=0.000163412946118649, ymax=1.31960669956248,
ymode=log,
ytick style={color=black}
]
\addplot [semithick, myBlue, mark=square, mark size=1, mark options={solid}]
table {%
0.5 0.472146873273859
0.25 0.138765701219376
0.125 0.0359163602592009
0.0625 0.0105278742652093
0.03125 0.00772360107739184
0.015625 0.0150940553567211
};
\addlegendentry{\scriptsize LOD$_\beta$ $\ell = 2$}
\addplot [semithick, myRed, mark=triangle, mark size=1, mark options={solid}]
table {%
0.5 0.81516840996814
0.25 0.279892337751307
0.125 0.0750212706745128
0.0625 0.0185594112082872
0.03125 0.00775111799090609
0.015625 0.0151669884635794
};
\addlegendentry{\scriptsize MLLOD$_\beta$ $\ell = 2$}
\addplot [semithick, myBlue, mark=square, mark size=2.5, mark options={solid}]
table {%
0.5 0.472146873273859
0.25 0.137687899644363
0.125 0.0351819258937354
0.0625 0.00929256237423465
0.03125 0.00256516091214096
0.015625 0.00604751380737312
};
\addlegendentry{\scriptsize LOD$_\beta$ $\ell = 3$}
\addplot [semithick, myRed, mark=triangle, mark size=3, mark options={solid}]
table {%
0.5 0.81516840996814
0.25 0.280627719414843
0.125 0.0759001244615391
0.0625 0.0192040999135074
0.03125 0.0052975937533368
0.015625 0.00614079974622849
};
\addlegendentry{\scriptsize MLLOD$_\beta$ $\ell = 3$}
\addplot [semithick, myBlue, mark=square, mark size=4, mark options={solid}]
table {%
0.5 0.472146873273859
0.25 0.137687899644363
0.125 0.0350662549400256
0.0625 0.00928117204025722
0.03125 0.00243104567519143
0.015625 0.00128135231877693
};
\addlegendentry{\scriptsize LOD$_\beta$ $\ell = 4$}
\addplot [semithick, myRed, mark=triangle, mark size=5, mark options={solid}]
table {%
0.5 0.81516840996814
0.25 0.280627719414843
0.125 0.075979465842555
0.0625 0.0192783919034896
0.03125 0.00541650705697636
0.015625 0.00205012304021143
};
\addlegendentry{\scriptsize MLLOD$_\beta$ $\ell = 4$}
\addplot [semithick, black, dashed]
table {%
	0.5 0.25
	0.25 0.0625
	0.125 0.015625
	0.0625 0.00390625
	0.03125 0.0009765625
	0.015625 0.000244140625
};
\addlegendentry{\scriptsize order 2}
\end{axis}

\end{tikzpicture}
	}
	\scalebox{.85}{
\begin{tikzpicture}

\begin{axis}[
legend cell align={left},
legend style={fill opacity=0.8, draw opacity=1, text opacity=1, at={(0.03,0.97)}, anchor=north west, draw=white!80!black},
legend pos = south east,
log basis x={10},
log basis y={10},
tick align=outside,
tick pos=left,
x grid style={white!69.0196078431373!black},
xlabel={mesh size \(\displaystyle H\)},
xmin=0.0131390064883393, xmax=0.594603557501361,
xmode=log,
xtick style={color=black},
y grid style={white!69.0196078431373!black},
ylabel={rel error in $\|\cdot\|_{\infty,1}$},
ymin=0.000163412946118649, ymax=1.31960669956248,
ymode=log,
ytick style={color=black}
]
\addplot [thick, myBlue, mark=square, mark size=4, mark options={solid}]
table {%
0.5 0.475085210904259
0.25 0.138237817068249
0.125 0.0420004998712294
0.0625 0.0219282661443159
0.03125 0.0141308834840253
0.015625 0.00959849692864377
};
\addlegendentry{\scriptsize LOD $\ell = 4$}
\addplot [thick, myRed, mark=triangle, mark size=4, mark options={solid}]
table {%
0.5 0.814873451061285
0.25 0.278348394776289
0.125 0.0753122705692571
0.0625 0.0222448747765649
0.03125 0.0149059198508067
0.015625 0.0103365053852925 
};
\addlegendentry{\scriptsize MLLOD $\ell = 4$}
\addplot [thick, myGreen, mark=o, mark size=4, mark options={solid}]
table {%
0.5 0.678171061268977
0.25 0.489457068209658
0.125 0.445005727623949
0.0625 0.430955695319444
0.03125 0.402454305704023
0.015625 0.203305840030885 
};
\addlegendentry{\scriptsize FEM}
\addplot [semithick, black, dashed]
table {%
	0.5 0.5
	0.25 0.25
	0.125 0.125
	0.0625 0.0625
	0.03125 0.03125
	0.015625 0.015625
};
\addplot [semithick, black, dashed]
table {%
	0.5 0.25
	0.25 0.0625
	0.125 0.015625
	0.0625 0.00390625
	0.03125 0.0009765625
	0.015625 0.000244140625
};
\addlegendentry{\scriptsize order 1 and 2}
\end{axis}

\end{tikzpicture}
	}
	\caption{\small Relative errors for Example~3 with respect to $H$ for different spatial discretisation approaches. Left: Correctly weighted LOD approach and its lumped version. Right: Classical finite element approximation and naively weighted LOD approach (both lumped and non-lumped).}
	\label{fig:ex3}
\end{figure}
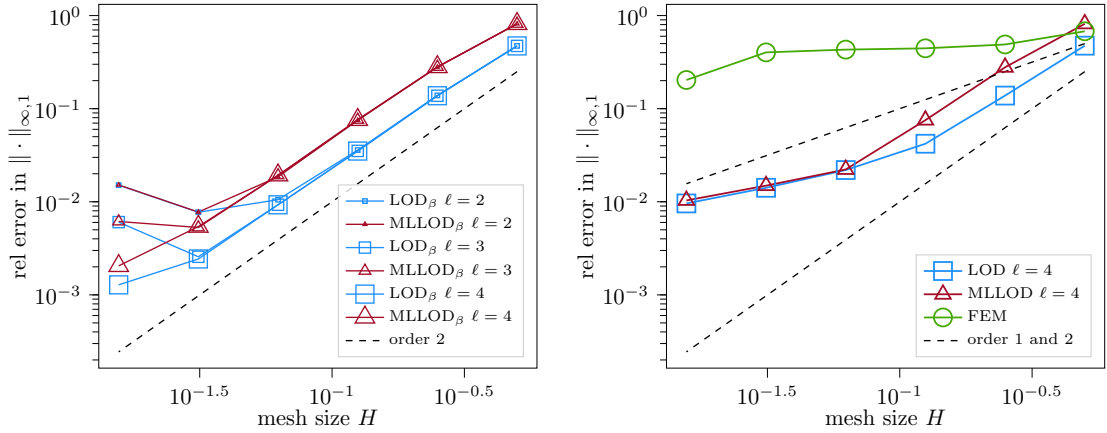


\begin{thebibliography}{10}
	
	\bibitem{AbdG11}
	A.~Abdulle and M.~J. Grote.
	\newblock Finite element heterogeneous multiscale method for the wave equation.
	\newblock {\em Multiscale Model. Simul.}, 9(2):766--792, 2011.
	
	\bibitem{AbdGS14}
	A.~Abdulle, M.~J. Grote, and C.~Stohrer.
	\newblock Finite element heterogeneous multiscale method for the wave equation:
	long-time effects.
	\newblock {\em Multiscale Model. Simul.}, 12(3):1230--1257, 2014.
	
	\bibitem{AbdH17}
	A.~Abdulle and P.~Henning.
	\newblock Localized orthogonal decomposition method for the wave equation with
	a continuum of scales.
	\newblock {\em Math. Comp.}, 86(304):549--587, 2017.
	
	\bibitem{AbdH17b}
	A.~Abdulle and P.~Henning.
	\newblock Multiscale methods for wave problems in heterogeneous media.
	\newblock In {\em Handbook of numerical methods for hyperbolic problems},
	volume~18 of {\em Handb. Numer. Anal.}, pages 545--576.
	Elsevier/North-Holland, Amsterdam, 2017.
	
	\bibitem{AltHP21}
	R.~Altmann, P.~Henning, and D.~Peterseim.
	\newblock Numerical homogenization beyond scale separation.
	\newblock {\em Acta Numer.}, 30:1--86, 2021.
	
	\bibitem{ArjR14}
	D.~Arjmand and O.~Runborg.
	\newblock Analysis of heterogeneous multiscale methods for long time wave
	propagation problems.
	\newblock {\em Multiscale Model. Simul.}, 12(3):1135--1166, 2014.
	
	\bibitem{baker1976}
	G.~A. Baker and V.~A. Dougalis.
	\newblock The effect of quadrature errors on finite element approximations for
	second order hyperbolic equations.
	\newblock {\em SIAM J. Numer. Anal.}, 13(4):577--598, 1976.
	
	\bibitem{chin99}
	M.~J.~S. Chin-Joe-Kong, W.~A. Mulder, and M.~van Veldhuizen.
	\newblock Higher-order triangular and tetrahedral finite elements with mass
	lumping for solving the wave equation.
	\newblock {\em J. Eng. Mech.}, 35(4):405--426, 1999.
	
	\bibitem{Chr09}
	S.~H. Christiansen.
	\newblock Foundations of finite element methods for wave equations of {Maxwell}
	type.
	\newblock In {\em Applied wave mathematics}, pages 335--393. Springer, Berlin,
	Heidelberg, 2009.
	
	\bibitem{ChuEL14}
	E.~T. Chung, Y.~Efendiev, and W.~T. Leung.
	\newblock Generalized multiscale finite element methods for wave propagation in
	heterogeneous media.
	\newblock {\em Multiscale Model. Simul.}, 12(4):1691--1721, 2014.
	
	\bibitem{Cia78}
	P.~G. Ciarlet.
	\newblock {\em The Finite Element Method for Elliptic Problems}.
	\newblock North-Holland, Amsterdam, 1978.
	
	\bibitem{Cle75}
	P.~Cl{\'e}ment.
	\newblock Approximation by finite element functions using local regularization.
	\newblock {\em Revue fran{\c{c}}aise d'automatique, informatique, recherche
		op{\'e}rationnelle. Analyse num{\'e}rique}, 9(R2):77--84, 1975.
	
	\bibitem{cohen95}
	G.~Cohen, P.~Joly, and N.~Tordjman.
	\newblock Higher order triangular finite elements with mass lumping for the
	wave equation.
	\newblock In {\em Proceedings of the Third International Conference on
		Mathematical and Numerical Aspects of Wave Propagation}, pages 270--279. SIAM
	Philadelphia, 1995.
	
	\bibitem{cohen01}
	G.~C. Cohen, P.~Joly, J.~E. Roberts, and N.~Tordjman.
	\newblock Higher order triangular finite elements with mass lumping for the
	wave equation.
	\newblock {\em SIAM J. Numer. Anal.}, 38(6):2047--2078, 2001.
	
	\bibitem{cui17}
	T.~Cui, W.~Leng, D.~Lin, S.~Ma, and L.~Zhang.
	\newblock High order mass-lumping finite elements on simplexes.
	\newblock {\em Numer. Math.-Theory Me.}, 10(2):331--350, 2017.
	
	\bibitem{EE03}
	W.~E and B.~Engquist.
	\newblock The heterogeneous multiscale methods.
	\newblock {\em Commun. Math. Sci.}, 1(1):87--132, 2003.
	
	\bibitem{EE05}
	W.~E and B.~Engquist.
	\newblock The heterogeneous multi-scale method for homogenization problems.
	\newblock In {\em Multiscale methods in science and engineering}, volume~44 of
	{\em Lect. Notes Comput. Sci. Eng.}, pages 89--110. Springer, Berlin,
	Heidelberg, 2005.
	
	\bibitem{EngHR11}
	B.~Engquist, H.~Holst, and O.~Runborg.
	\newblock Multi-scale methods for wave propagation in heterogeneous media.
	\newblock {\em Commun. Math. Sci.}, 9(1):33--56, 2011.
	
	\bibitem{EngHR12}
	B.~Engquist, H.~Holst, and O.~Runborg.
	\newblock Multiscale methods for wave propagation in heterogeneous media over
	long time.
	\newblock In {\em Numerical analysis of multiscale computations}, volume~82 of
	{\em Lect. Notes Comput. Sci. Eng.}, pages 167--186. Springer, Berlin,
	Heidelberg, 2012.
	
	\bibitem{ErnG17}
	A.~Ern and J.-L. Guermond.
	\newblock Finite element quasi-interpolation and best approximation.
	\newblock {\em ESAIM Math. Model. Numer. Anal.}, 51(4):1367--1385, 2017.
	
	\bibitem{GalP15}
	D.~Gallistl and D.~Peterseim.
	\newblock Stable multiscale {P}etrov--{G}alerkin finite element method for high
	frequency acoustic scattering.
	\newblock {\em Comput. Methods Appl. Mech. Engrg.}, 295:1--17, 2015.
	
	\bibitem{geevers18}
	S.~Geevers, W.~A. Mulder, and J.~J.~W. van~der Vegt.
	\newblock New higher-order mass-lumped tetrahedral elements for wave
	propagation modelling.
	\newblock {\em SIAM J. Sci. Comput.}, 40(5):A2830--A2857, 2018.
	
	\bibitem{geevers19}
	S.~Geevers, W.~A. Mulder, and J.~J.~W. van~der Vegt.
	\newblock Efficient quadrature rules for computing the stiffness matrices of
	mass-lumped tetrahedral elements for linear wave problems.
	\newblock {\em SIAM J. Sci. Comput.}, 41(2):A1041--A1065, 2019.
	
	\bibitem{Geo08}
	E.~H. Georgoulis.
	\newblock Inverse-type estimates on $hp$-finite element spaces and
	applications.
	\newblock {\em Math. Comp.}, 77(261):201--219, 2008.
	
	\bibitem{GraHS05}
	I.~G. Graham, W.~Hackbusch, and S.~A. Sauter.
	\newblock Finite elements on degenerate meshes: inverse-type inequalities and
	applications.
	\newblock {\em IMA J. Numer. Anal.}, 25(2):379--407, 2005.
	
	\bibitem{HelK19}
	F.~Hellman and T.~Keil.
	\newblock Gridlod.
	\newblock \texttt{https://github.com/fredrikhellman/gridlod}, 2019.
	\newblock GitHub repository, commit 0ed4c096df75040145978d48c5307ef5678efed3.
	
	\bibitem{HenP13}
	P.~Henning and D.~Peterseim.
	\newblock Oversampling for the multiscale finite element method.
	\newblock {\em Multiscale Model. Simul.}, 11(4):1149--1175, 2013.
	
	\bibitem{HowW97}
	T.~Y. Hou and X.-H. Wu.
	\newblock A multiscale finite element method for elliptic problems in composite
	materials and porous media.
	\newblock {\em J. Comput. Phys.}, 134(1):169--189, 1997.
	
	\bibitem{JiaE12}
	L.~Jiang and Y.~Efendiev.
	\newblock A priori estimates for two multiscale finite element methods using
	multiple global fields to wave equations.
	\newblock {\em Numer. Meth. Part. D. E.}, 28(6):1869--1892, 2012.
	
	\bibitem{JiaEG10}
	L.~Jiang, Y.~Efendiev, and V.~Ginting.
	\newblock Analysis of global multiscale finite element methods for wave
	equations with continuum spatial scales.
	\newblock {\em Appl. Numer. Math.}, 60(8):862--876, 2010.
	
	\bibitem{Jol03}
	P.~Joly.
	\newblock Variational methods for time-dependent wave propagation problems.
	\newblock In {\em Topics in computational wave propagation}, volume~31 of {\em
		Lect. Notes Comput. Sci. Eng.}, pages 201--264. Springer, Berlin, Heidelberg,
	2003.
	
	\bibitem{komatitsch98}
	D.~Komatitsch and J.~P. Vilotte.
	\newblock The spectral element method: an efficient tool to simulate the
	seismic response of 2{D} and 3{D} geological structures.
	\newblock {\em B. Seismol. Soc. Am.}, 88(2):368--392, 1998.
	
	\bibitem{LioM72}
	J.~L. Lions and E.~Magenes.
	\newblock {\em Non-homogeneous boundary value problems and applications},
	volume~1.
	\newblock Springer Verlag, New York-Heidelberg, 1972.
	
	\bibitem{liu17}
	Y.~Liu, J.~Teng, T.~Xu, and J.~Badal.
	\newblock Higher-order triangular spectral element method with optimized
	cubature points for seismic wavefield modeling.
	\newblock {\em J. Comput. Phys.}, 336:458--480, 2017.
	
	\bibitem{Mai20}
	R.~Maier.
	\newblock {\em Computational Multiscale Methods in Unstructured Heterogeneous
		Media}.
	\newblock PhD thesis, University of Augsburg, 2020.
	
	\bibitem{MaiP19}
	R.~Maier and D.~Peterseim.
	\newblock Explicit computational wave propagation in micro-heterogeneous media.
	\newblock {\em BIT Numer. Math.}, 59(2):443--462, 2019.
	
	\bibitem{MaiV20}
	R.~Maier and B.~Verf{\"u}rth.
	\newblock Multiscale scattering in nonlinear {Kerr}-type media.
	\newblock {\em ArXiv Preprint}, 2011.09168, 2020.
	
	\bibitem{MalP14}
	A.~M{\aa}lqvist and D.~Peterseim.
	\newblock Localization of elliptic multiscale problems.
	\newblock {\em Math. Comp.}, 83(290):2583--2603, 2014.
	
	\bibitem{MalP20}
	A.~M{\aa}lqvist and D.~Peterseim.
	\newblock {\em Numerical homogenization by localized orthogonal decomposition},
	volume~5 of {\em SIAM Spotlights}.
	\newblock Society for Industrial and Applied Mathematics (SIAM), Philadelphia,
	PA, 2020.
	
	\bibitem{mulder96}
	W.~A. Mulder.
	\newblock A comparison between higher-order finite elements and finite
	differences for solving the wave equation.
	\newblock In {\em Proceedings of the Second ECCOMAS Conference on Numerical
		Methods in Engineering}, pages 344--350. John Wiley \& Sons, 1996.
	
	\bibitem{mulder13}
	W.~A. Mulder.
	\newblock New triangular mass-lumped finite elements of degree six for wave
	propagation.
	\newblock {\em Prog. Electromagn. Res., (141)}, pages 671--692, 2013.
	
	\bibitem{OwhZ08}
	H.~Owhadi and L.~Zhang.
	\newblock Numerical homogenization of the acoustic wave equations with a
	continuum of scales.
	\newblock {\em Comput. Methods Appl. Mech. Engrg.}, 198(3-4):397--406, 2008.
	
	\bibitem{OwhZ17}
	H.~Owhadi and L.~Zhang.
	\newblock Gamblets for opening the complexity-bottleneck of implicit schemes
	for hyperbolic and parabolic {ODE}s/{PDE}s with rough coefficients.
	\newblock {\em J. Comput. Phys.}, 347:99--128, 2017.
	
	\bibitem{OwhZB14}
	H.~Owhadi, L.~Zhang, and L.~Berlyand.
	\newblock Polyharmonic homogenization, rough polyharmonic splines and sparse
	super-localization.
	\newblock {\em ESAIM Math. Model. Numer. Anal.}, 48(2):517--552, 2014.
	
	\bibitem{Pet17}
	D.~Peterseim.
	\newblock Eliminating the pollution effect in {H}elmholtz problems by local
	subscale correction.
	\newblock {\em Math. Comp.}, 86(305):1005--1036, 2017.
	
	\bibitem{PetS17}
	D.~Peterseim and M.~Schedensack.
	\newblock Relaxing the {CFL} condition for the wave equation on adaptive
	meshes.
	\newblock {\em J. Sci. Comput.}, 72(3):1196--1213, 2017.
	
	\bibitem{PetV20}
	D.~Peterseim and B.~Verf{\"u}rth.
	\newblock Computational high frequency scattering from high-contrast
	heterogeneous media.
	\newblock {\em Math. Comp.}, 89(326):2649--2674, 2020.
	
	\bibitem{Sch98}
	Ch. Schwab.
	\newblock {\em {$p$}- and {$hp$}-finite element methods. Theory and
		applications in solid and fluid mechanics}.
	\newblock Numerical Mathematics and Scientific Computation. The Clarendon
	Press, Oxford University Press, New York, 1998.
	
	\bibitem{seriani94}
	G.~Seriani and E.~Priolo.
	\newblock Spectral element method for acoustic wave simulation in heterogeneous
	media.
	\newblock {\em Finite Elem. Anal. Des.}, 16(3-4):337--348, 1994.
	
\end{thebibliography}
\end{document}